\documentclass[twoside,11pt]{article}

\usepackage[english]{babel}
\usepackage{graphicx}
\graphicspath{{../Images_Lasso/}}

\usepackage{amsmath,amssymb,amsthm}
\usepackage{hyperref}
\usepackage[round]{natbib}
\usepackage{xcolor}
\usepackage{url}

\newcommand\R{\mathbb{R}}

\newcommand\Hb{\mathbb{H}}
\newcommand\n[1]{\left\|#1\right\|}
\newcommand\ps[2]{\langle#1,#2\rangle}
\newcommand\bbeta{\boldsymbol\beta}
\newcommand\estim{\widehat{\boldsymbol\beta}_{\boldsymbol\lambda,\infty}}
\newcommand\cov{{\boldsymbol\Gamma}}
\newcommand\covhat{{\widehat{\boldsymbol\Gamma}}}

\newtheorem{theorem}{Theorem}
\newtheorem{lemma}{Lemma}
\newtheorem{proposition}{Proposition}
\newtheorem{corollary}{Corollary}

\begin{document}

\title{Lasso in infinite dimension: application to variable selection in functional multivariate linear regression}
\author{Angelina Roche \url{roche@ceremade.dauphine.fr} }


\maketitle

\begin{abstract} 
It is more and more frequently the case in applications that the data we observe come from one or more random variables taking values in an infinite dimensional space, e.g. curves. The need to have tools adapted to the nature of these data explains the growing interest in the field of functional data analysis. The model we study in this paper assumes a linear dependence between a quantity of interest and several covariates, at least one of which has an infinite dimension. To select the relevant covariates in this context, we investigate adaptations of the Lasso method. Two estimation methods are defined. The first one consists in the minimization of a Group-Lasso criterion on the multivariate functional space $\mathbf H$. The second one minimizes the same criterion but on a finite dimensional subspaces of $\mathbf H$ whose dimension is chosen by a penalized least squares method. We prove oracle inequalities of sparsity in the case where the design is fixed or random. To compute the solutions of both criteria in practice, we propose a coordinate descent algorithm. A numerical study on simulated and real data illustrates the behavior of the estimators. 
\end{abstract}


$\;$

%
%
\section{Introduction}
More and more often, the data we observe come from one or more random variables taking their values in a space of infinite dimension. This is the case, for example, for data that can be represented as curves. The need to develop tools adapted to the nature of the data explains the growing interest in the field of functional data analysis \citep{RS05,FV06,FR11}. It has proven to be very fruitful in many applications, for example in spectrometry \citep[see for example][]{Pham+10}, in the study of electroencephalograms \citep{di_multilevel_2009}, in biomechanics \citep{sorensen_quantification_2012} and in econometrics \citep{laurini_dynamic_2014}. 

In some contexts, and more and more often, the data are a finite number of curves. We call this case multidimensional functional data. This is the case in \citet{ACEV04} where the objective is to predict the ozone concentration of the next day from the ozone concentration curve, the $NO$ concentration curve, the $NO_2$ concentration curve, the wind speed curve and the wind direction of the current day. Another example comes from nuclear safety problems where the risk of failure of a nuclear reactor vessel in case of a loss of coolant accident is studied as a function of the evolution of the temperature, pressure and heat transfer parameter in the vessel \citep{Roche15}. It can also happen, perhaps more often, that the observed quantities are of different natures (curves and vectors or scalars). This case has motivated the study of partial linear models (see for example \citealt{Shin09,WSW21,XU202044}) where a quantity of interest $Y$ depends both on vectors and on functional covariates. 

In the case where the number of covariates, especially functional or infinite dimensional covariates, is large, it may be necessary to select the most relevant covariates for prediction, either to solve the computational problems posed by the complexity of the data or to obtain an interpretable prediction procedure. 

The objective of this paper is to study the link between a real response $Y$ and a vector of covariates $\mathbf X=(X^1,...,X^p)$ which can be of different nature (curves or vectors or scalar quantities). 
We assume that, for all $j=1,...,p$, $i=1,...,n$, $X^j_i\in \Hb_j$ where $(\Hb_j,\n{\cdot}_j,\ps\cdot\cdot_j)$ is a separable Hilbert space. Our covariate $\{\mathbf X_i\}_{1\leq i\leq n}$ is then in the product space $\mathbf{H}=\Hb_1\times...\times\Hb_p$, which is also a separable Hilbert space equipped with its natural scalar product 
$$\ps{\mathbf f}{\mathbf g} = \sum_{j=1}^p \ps{f_j}{g_j}_j \text{ for all }\mathbf f=(f_1,...,f_p), \mathbf g=(g_1,...,g_p)\in\mathbf{H}$$
and usual norm $\n{\mathbf f}=\sqrt{\ps{\mathbf f}{\mathbf f}}$. 

We suppose that our observations follow the \emph{multivariate functional linear model},
\begin{equation}\label{def:model}
 Y_i = \sum_{j=1}^p\ps{\beta_j^*}{X_i^j}_j +\varepsilon_i = \ps{\boldsymbol\beta^*}{\mathbf X_i}+\varepsilon_i,
 \end{equation}
where, $\boldsymbol\beta^*=(\boldsymbol\beta_1^*,...,\boldsymbol\beta_p^*)\in\mathbf H$ is unknown and $\left\{\varepsilon_i\right\}_{1\leq i\leq n}\sim_{i.i.d.}\mathcal N(0,\sigma^2)$. The covariates $\{\mathbf X_i\}_{1\leq i\leq n}$ can be either fixed elements of $\mathbf H$ (fixed design) or i.i.d centered random variables in $\mathbf H$ (random design) independent of $\left\{\varepsilon_i\right\}_{1\leq i\leq n}$. 

Note that our model does not require the $\mathbb H_j$'s to be functional spaces, we can have $\mathbb H_j=\mathbb R$ or $\mathbb H_j=\mathbb R^d$, for some $j\in\{1,...,p\}$. The case where $\mathbb H_j=\mathbb R$, for all $j=1,\hdots,p$ exactly corresponds to the classical multivariate regression model. 

The functional linear model, which corresponds to the case $p=1$ in the equation~\eqref{def:model}, has been widely studied. It has been defined by \citet{CFS1999} who proposed an estimator based on principal component analysis. Splines estimators have also been proposed by \citet{RD91,CFS2003,CKS09} as well as estimators based on the decomposition of the slope function $\boldsymbol\beta$ in the Fourier domain \citep{RS05,LH07,CJ10} or in a general basis \citep{CJ10b,CJ12}. In a similar context, we also mention the work of \citet{KM14} on Lasso. In this paper, it is assumed that the function $\boldsymbol\beta$ is well represented as a sum of a small number of well separated spikes. In the case where $p=2$, $\mathbb H_1$ a functional space and $\mathbb H_2=\mathbb R^d$, the model~\eqref{def:model} is called \emph{partial functional linear regression model} and has been studied for example by \citet{Shin09,SL12} who proposed principal component regression and ridge regression approaches for the estimation of the two coefficients of the model. 

Few works have been devoted to the multivariate functional linear model which corresponds to the case where $p\geq 2$ and the $\mathbb H_j$ are function spaces for all $j=1,\hdots,p$. To the best of our knowledge, the model was first mentioned in the work of \citet{CCS07} under the name \emph{multiple functional linear model}. An estimator of $\boldsymbol\beta$ is defined with an iterative backfitting algorithm and applied to the ozone prediction dataset initially studied by \citet{ACEV04}. Variable selection is performed by testing all possible models and selecting the one that minimizes the prediction error on a test sample. Let us also mention the work of \citet{CYC16} who consider a multivariate linear regression model with functional output. They define a consistent and asymptotically normal estimator based on the multivariate functional principal components initially proposed by \citet{CCY14}. 


In the case where the covariates are finite dimensional and $p$ is large, the usual approach to select variables is to use a penalty of type $\ell_1$. This case has been widely studied, with many variations and improvements. One of the most common variable selection methods, the Lasso \citep{Tibshirani96,CDS98}, consists of the minimization of a least squares criterion with an $\ell_1$ penalty. The statistical properties of the Lasso estimator are now well understood. Sparsity oracle inequalities have been obtained for predictive losses in particular in standard multivariate or nonparametric regression models \citep[see for example][]{BTW07,bickel_simultaneous_2009,Koltchinskii09,BPR11}. 

The Group-Lasso \citep{YL2006,CH08} addresses the case where the set of covariates can be partitioned into a number of groups. To take into account the group structure in the data, our model can be rewritten as $\mathbb H_j=\mathbb R^{d_j}$, $j=1,\hdots,p$, where $p$ is the number of groups and $d_j$ is the cardinal of the $j$-th group. \citet{HZ10} show that, under certain conditions called \emph{strong group sparsity}, the Group-Lasso penalty is more efficient than the Lasso penalty. \citet{lounici_oracle_2011} proved oracle inequalities for the prediction and $\ell_2$ estimation error that are optimal in the minimax sense. Their theoretical results also demonstrate that Group-Lasso can improve Lasso in prediction and estimation.  \citet{vandeGeer14} proved sharp oracle inequalities for general weakly decomposable regularization penalties, including Group-Lasso penalties.  This approach has proven fruitful in many settings such as time series \citep{CYZ14}, generalized linear models \citep{BLG14} in particular Poisson regression~\citep{IPR16} or logistic regression \citep{MGB08, Kwemou16}, the study of panel data \citep{DQS16}, the prediction of breast or prostate cancer \citep{FWYPLOPB16,ZCJML16}. The theoretical results were extended to the case where the errors are heteroscedastic by \citet{DHMS13}.

Drawing inspiration from \citet{lounici_oracle_2011} we define two criteria
\begin{equation}\label{eq:LASSOinf}
\widehat{\boldsymbol\beta}_{\boldsymbol\lambda,\infty}\in{\arg\min}_{\boldsymbol\beta=(\beta_1,...,\beta_p)\in\mathbf H}\left\{\frac1n\sum_{i=1}^n\left(Y_i-\ps{\boldsymbol\beta}{\mathbf X_i}\right)^2+2\sum_{j=1}^p\lambda_j\n{\beta_j}_j\right\},
\end{equation}
and
\begin{equation}\label{eq:LASSOproj}
\widehat{\boldsymbol\beta}_{\boldsymbol\lambda,m}\in{\arg\min}_{\boldsymbol\beta=(\beta_1,...,\beta_p)\in\mathbf H^{(m)}}\left\{\frac1n\sum_{i=1}^n\left(Y_i-\ps{\boldsymbol\beta}{\mathbf X_i}\right)^2+2\sum_{j=1}^p\lambda_j\n{\beta_j}_j\right\},
\end{equation}
where $\boldsymbol\lambda=(\lambda_1,...,\lambda_p)$ are positive parameters and $(\mathbf H^{(m)})_{m\geq 1}$ is a sequence of nested finite-dimensional subspaces of $\mathbf H$, to be specified later. 

The case where the product space $\mathbf H$ is of finite dimension has been widely treated (see the references above). However, few papers deal with the infinite-dimensional case. Most of the literature in functional data analysis naturally focuses on dimension reduction methods (mainly projection onto a spline basis or onto the principal component basis in~\citealt{RS05,FR11}) to reduce data complexity. More recently, clustering approaches have been considered \citep[see for example][]{D17} as well as variable selection methods using $\ell^1$ penalties. \citep{KXYZ16} have proposed a Lasso type penalty allowing to select the Karhunen-Loève coefficients of the functional variable simultaneously with the coefficients of the vector variable in the partial functional linear model (case $p=2$, $\Hb_1=\mathbb L^2(T)$, $\Hb_2=\R^d$ of the Model~\eqref{def:model}). Group-Lasso and adaptive Group-Lasso procedures have been proposed by \citet{AV14,AV16} to select the important observation points $t_1,...,t_n$ (\emph{impact points}) in a regression model where the covariates are the discrete values $(X(t_1),...,X(t_p))$ of a random function $X$.  Bayesian approaches have been proposed by \citet{GABP19} in the case where the $\beta_j^*$ are sparse step functions. The natural extension of the approaches developed in the field of functional data analysis in our context leads to the projected version of the criterion defined in equation~\eqref{eq:LASSOproj} with $\mathbf H^{(m)}$ generated by a multivariate splines basis or an fPCA basis. However, the projection step induces a bias that must be taken into account.   

Some recent contributions~\citep [see for example][]{GV16,S18} emphasize the need to work at the interface between high-dimensional statistics, functional data analysis, and machine learning to deal more effectively with specific problems of high-dimensional or infinite-dimensional data. Indeed, the problem of infinite-dimensional variable selection is also considered in the machine learning community, especially in the context of multiple-kernel learning. \citet{Bach08,NR08} proved the consistency of model estimation and selection, as well as prediction and estimation bounds for the Group-Lasso estimator, when the data belong to Reproducing Kernel Hilbert Spaces. In these papers, the criterion is minimized on the whole product space $\mathbf H$, leading to \eqref{eq:LASSOinf}. However, imposing that the data be in a Reproducing Kernel Hilbert Space is too restrictive in the domain of functional data because it implies a constraint on the unknown regularity of the data. To the best of our knowledge, the theoretical study of \eqref{eq:LASSOinf} has not been done when the data are in a general Hilbert space.  

Our approach also covers the case where $Y_i$ depends on a single functional variable $Z_i:T\to\mathbb R$ and we want to determine whether observing the entire curve $\{Z_i(t), t\in T\}$ is useful to predict $Y_i$ or whether it is sufficient to observe it on some subsets of $T$. For this purpose, we define $T_1,\hdots,T_p$ a partition of the set $T$ into subintervals and consider the restrictions $X_i^j:T_j\to\mathbb R$ of $Z_i$ on $T_j$. If the corresponding coefficient $\beta_j^*$ is zero, we know that $X_i^j$ is, a priori, irrelevant to predict $Y_i$ and, therefore, that the behavior of $Z_i$ on the interval $T_j$ has no significant influence on $Y_i$. The idea of using a Lasso type criterion or a Danzig selector in this context, called the FLIRTI method (for Functional LInear Regression That is Interpretable) has been developed by \citet{JWZ09}.



\subsection*{Contribution of the paper}

The properties of the solution of the Group-Lasso problem \eqref{eq:LASSOinf} have been studied for example by \cite{lounici_oracle_2011} under restricted eigenvalue type assumptions in the finite-dimensional case.  \cite{BT17} have improved these results by obtaining sharp versions of the sparsity oracle inequalities. The aim of this paper is to study the case where $\dim(\mathbf H)=+\infty$ and to answer the following questions: are we able to obtain sharp oracle inequalities when $\dim(\mathbf H)=+\infty$? How to compute the solution of a Lasso problem in this infinite-dimensional context ?

To answer the first question, we must first define a restricted eigenvalue condition (or an equivalent). Unfortunately, the question of the restricted eigenvalue assumption in an infinite dimensional space turns out to be a complex issue. Indeed, we first prove in Section~\ref{sec:RE} that no such hypothesis can be verified on the entire space $\mathbf H$ in infinite dimension, or even when the data dimension is too large.  We consider as an alternative, the minimal ratio $\tilde\kappa_n^{(m)}(s)$ between the empirical norm and the norm of $\mathbb H$ on the cone 
$$
\{\boldsymbol\delta\in\mathbf H^{(m)}, \exists J\subset\{1,\hdots,p\}, |J|\leq s, \sum_{j\notin J}\lambda_j\n{\delta_j}_j\leq 3\sum_{j\in J}\lambda_j\n{\delta_j}_j\}. 
$$
This quantity, supposed to be constant in finite dimension in the works of \cite{lounici_oracle_2011,BT17}, is seen here as a sequence which decreases towards 0 when $m=\dim(\mathbf H^{(m)})$ increases, at a rate which will determine the convergence rate of the final estimator. This rate of convergence thus plays the role of a regularity parameter. This is, to our knowledge, a new approach to the problem.

We prove in section~\ref{sec:ineg_oracle_empirical} a sharp oracle inequality for both criteria~\eqref{eq:LASSOinf} and~\eqref{eq:LASSOproj} without any assumption other than noise normality. The proofs and results are similar to those of \cite{lounici_oracle_2011,BT17} except that we have to deal with the remaining term due to the violation of the restricted eigenvalue assumption for the solution of~\eqref{eq:LASSOinf} and the bias due to the projection for the solution of~\eqref{eq:LASSOproj}. The results are true for both fixed and random designs. We find, as expected, that the properties of the projected estimator~\eqref{eq:LASSOproj} depend strongly on the choice of the projection dimension $m$. A data-driven criterion for selecting the dimension $m$, inspired by the work of~\cite{BBM99} and their adaptation to the functional linear model by~\cite{BMR16}, is proposed.
In Section~\ref{sec:ineg_oracle_prediction}, we obtain a sparsity oracle inequality for the theoretical prediction error under certain assumptions of sub-Gaussianity of the data distribution. 
The sections~\ref{sec:algo} and \ref{sec:simus} are devoted to the numerical properties of the solution. If the solution of the criterion~\eqref{eq:LASSOproj} can be computed directly from the coefficients of the data in a basis of the space $\mathbf  H^{(m)}$ with tools dedicated to multivariate data, it is not the same for the solution of the criterion~\eqref{eq:LASSOinf} which requires solving an infinite dimensional optimization problem. We then define a computational algorithm allowing to minimize the criterion~\eqref{eq:LASSOinf} directly in the space $\mathbf H$, without projecting the data. This computational algorithm is also used to solve the criterion~\eqref{eq:LASSOproj} to facilitate comparisons. The properties of the estimators are studied numerically in section~\ref{sec:simus} on simulated data sets. We then applied both estimation procedures to the prediction of energy consumption of household appliances. 

\subsection*{Notations}
Throughout the paper, we denote, for all ${J\subseteq\{1,...,p\}}$ the sets 
$$\mathbf H_J:=\prod_{j\in J}\mathbb H_j.$$
Consider that the data $\mathbf X_1,\hdots,\mathbf X_n$ has been centered, we also define
$$\widehat{\boldsymbol{\boldsymbol\Gamma}}:\boldsymbol\beta\in\mathbf H\mapsto \frac1n\sum_{i=1}^n\ps{\boldsymbol\beta}{\mathbf X_i}\mathbf X_i,$$
the empirical covariance operator associated to the data and its restricted versions 
$$\widehat{\boldsymbol{\boldsymbol\Gamma}}_{J,J'}:\boldsymbol\beta=(\beta_j, j\in J)\in\mathbf H_J\mapsto\left(\frac1n\sum_{i=1}^n\sum_{j\in J}\ps{\beta_j}{X_i^j}_jX_i^{j'}\right)_{j'\in J'}\in\mathbf H_{J'},$$
defined for all $J,J'\subseteq\{1,...,p\}$.  
For simplicity, we also denote $\widehat{\boldsymbol{\boldsymbol\Gamma}}_J:=\widehat{\boldsymbol{\boldsymbol\Gamma}}_{J,J}$, 
 $\widehat{\boldsymbol{\boldsymbol\Gamma}}_{J,j}:=\widehat{\boldsymbol{\boldsymbol\Gamma}}_{J,\{j\}}$ and $\widehat\Gamma_j:=\widehat{\boldsymbol\Gamma}_{\{j\},\{j\}}$.


For $\boldsymbol\beta=(\boldsymbol\beta_1,...,\boldsymbol\beta_p)\in\mathbf H$, we denote by $J(\boldsymbol\beta):=\{j,\ \beta_j\neq 0\}$ the support of $\boldsymbol\beta$ and $|J(\boldsymbol\beta)|$ its cardinality. 

We also denote by $\mathbb P_{\mathbf X}(\cdot)=\mathbb P(\cdot|\mathbf X_1,\hdots,\mathbf X_n)$ the conditional probability with respect to the design if it is random or $\mathbb P_{\mathbf X}(\cdot)=\mathbb P$ if the design is fixed. 

\section{Discussion on the restricted eigenvalues assumption}
\label{sec:RE}

\subsection{The restricted eigenvalues assumption does not hold if $\dim(\mathbf H)=+\infty$}
\label{sec:REinf}

Sparsity oracle inequalities are usually obtained under conditions on the design matrix. One of the most common is the restricted eigenvalues property  \citep{bickel_simultaneous_2009,lounici_oracle_2011}. Translated to our context, this assumption may be written as follows.

$(A_{RE(s)})$: There exists a positive number $\kappa=\kappa(s)$ such that
$$\min\left\{\frac{\n{\boldsymbol \delta}_n}{\sqrt{\sum_{j\in J}\n{\delta_j}_j^2}}, |J|\leq s, \boldsymbol\delta=(\delta_1,...,\delta_p)\in\mathbf H\backslash\{0\}, \sum_{j\notin J}\lambda_j\n{\delta_j}_j\leq c_0\sum_{j\in J}\lambda_j\n{\delta_j}_j\right\}\geq\kappa,$$
with $\n{f}_n:=\sqrt{\frac1n\sum_{i=1}^n \ps{f}{\mathbf X_i}^2}$ the empirical norm on $\mathbf H$ naturally associated with our problem. 

As explained in \citet[Section 3]{bickel_simultaneous_2009}, this assumption can be seen as a ''positive definiteness'' condition on the Gram matrix restricted to sparse vectors. In the finite dimensional context, \citet{VB09}
 have proven that this condition covers a large class of design matrices. 
 
 The next lemma, proven in Section~\ref{subsec:proofREfalse}, shows that this assumption does not hold when $\dim(\mathbf H_J)$ is too large for a subset $J$ of $\{1,\hdots,p\}$.
  
 \begin{lemma}\label{lem:REfalse}
 Suppose that there exists $J\subset\{1,\hdots,p\}$ such that $\dim(\mathbf H_J)>{\rm rk}(\widehat{\boldsymbol\Gamma}_J)$, then, for all $s\geq|J|$, for all $c_0>0$
 \begin{equation*}
 \small\min\left\{\frac{\n{\boldsymbol \delta}_n}{\sqrt{\sum_{j\in J}\n{\delta_j}_j^2}}, |J|\leq s, \boldsymbol\delta=(\delta_1,...,\delta_p)\in\mathbf H\backslash\{0\}, \sum_{j\notin J}\lambda_j\n{\delta_j}_j\leq c_0\sum_{j\in J}\lambda_j\n{\delta_j}_j\right\}=0. 
 \end{equation*}
 \end{lemma}

Remark that, since ${\rm Im}(\widehat{\boldsymbol\Gamma}_J)={\rm span}\{(\mathbf X_i^j)_{j\in J},i=1,\hdots,n\}$, ${\rm rk}(\widehat{\boldsymbol\Gamma}_J)\leq n$. Then the condition $\dim(\mathbf H_J)>{\rm rk}(\widehat{\boldsymbol\Gamma}_J)$ is unfortunately always verified if $\dim(\mathbf H)=+\infty$.
\subsection{Finite-dimensional subspaces and restriction of the restricted eigenvalues assumption} 

 The infinite-dimensional nature of the data is the main obstacle here. To circumvent the dimensionality problem, we restrict the assumption to finite-dimensional spaces. 
In the sequel, we focus on spaces spanned by the $m$-first elements of an orthonormal basis $(\boldsymbol\varphi^{(k)})_{k\geq 1}$ i.e. $\mathbf H^{(m)}:={\rm span}\left\{\boldsymbol\varphi^{(1)},\hdots,\boldsymbol\varphi^{(m)}\right\}$. To obtain sparsity oracle inequalities, we suppose fulfilled the following condition of support compatibility of the basis. We denote by 
 $$
 \pi_j : \mathbf f=(f_1,\hdots,f_p)\in\mathbf H\mapsto (0,\hdots,0,f_j,0,\hdots,0)
 $$
 the projection operator into the $j$-th coordinates and 
 \[
 \Pi_m: \mathbf f\in\mathbf H\mapsto \sum_{j=1}^m\langle\mathbf f,\boldsymbol\varphi^{(k)}\rangle\boldsymbol\varphi^{(k)}
 \]
 the projection operator into $\mathbf H^{(m)}$.
 
 \bigskip
 
 
\noindent $(C_{supp})$
 
  For all $m\geq 1$, $j\in\{1,\hdots,p\}$, the operator $\Pi_m$ commutes with $\pi_j$. 
  
  \bigskip

Condition $(C_{supp})$ appears necessary to obtain the sparsity oracle inequality. It is a condition on the basis $(\boldsymbol\varphi^{(k)})_{k\geq 1}$. It is the case for instance if, for all $k\geq 1$, $|J(\boldsymbol\varphi^{(k)})|=1$. Indeed, this means that $\pi_j\boldsymbol\varphi^{(k)}=\mathbf 1_{\{j\in J(\boldsymbol\varphi^{(k)})\}}\boldsymbol\varphi^{(k)}$ for all $\mathbf f\in\mathbf H$ 
\[
\pi_j\Pi_m \mathbf f = \pi_j\sum_{k=1}^m\langle\mathbf f,\boldsymbol\varphi^{(k)}\rangle \boldsymbol\varphi^{(k)}=\sum_{k=1}^m\langle\mathbf f,\boldsymbol\varphi^{(k)}\rangle \pi_j\boldsymbol\varphi^{(k)}=\sum_{k=1}^m\mathbf 1_{\{j\in J(\boldsymbol\varphi^{(k)})\}}\langle f_j,\varphi_j^{(k)}\rangle_j \boldsymbol\varphi^{(k)}
\]
and we deduce that
\[
\Pi_m\pi_j\mathbf f = \sum_{k=1}^m\langle\pi_j\mathbf f,\boldsymbol\varphi^{(k)}\rangle\boldsymbol\varphi^{(k)}=\sum_{k=1}^m\langle f_j,\varphi^{(k)}_j\rangle_j\boldsymbol\varphi^{(k)}=\pi_j\Pi_m\mathbf f. 
\]

It is possible now to define a restricted eigenvalues property on the projection on the data on the finite-dimensional space $\mathbf H^{(m)}$. We would like to emphasize first that the viewpoint is different. In finite-dimensional contexts (see e.g. \citealt{bickel_simultaneous_2009,lounici_oracle_2011}), the restricted eigenvalue property is an assumption on the design matrix. In infinite-dimensional contexts, it seems more natural, since, there is no \emph{a priori} dimension for the data, to define a sequence $(\tilde\kappa_n^{(m)})_{m\geq 1}$ depending on the sparsity level $s\in\{1,\hdots,p\}$ as follows
 \begin{eqnarray}
 \tilde\kappa_n^{(m)}(s)&:=&\nonumber\\
 &&\hspace{-2cm}\small\min\left\{\frac{\n{\boldsymbol \delta}_n}{\sqrt{\sum_{j\in J}\n{\delta_j}_j^2}}, |J|\leq s, \boldsymbol\delta=(\delta_1,...,\delta_p)\in\mathbf H^{(m)}\backslash\{0\}, \sum_{j\notin J}\lambda_j\n{\delta_j}_j\leq 3\sum_{j\in J}\lambda_j\n{\delta_j}_j\right\}. \label{eq:tildekappa}
 \end{eqnarray}
The quantities $\tilde\kappa_n^{(m)}(s)$ are linked with the spectral radius of restrictions of the empirical covariance operator $\covhat$ by the following relationship
 \begin{equation}\label{eq:encadrement_tildekappa}
 \min_{J\subseteq\{1,\hdots,p\}; |J|\leq s} \rho\left(\widehat{\boldsymbol\Gamma}_{J|m}^{-1/2}\right)^{-1}\geq\tilde\kappa_n^{(m)}(s)\geq \rho\left(\widehat{\boldsymbol\Gamma}_{m}^{-1/2}\right)^{-1},
 \end{equation}
 where $\widehat{\boldsymbol\Gamma}_{J|m}=\left(\langle\widehat{\boldsymbol\Gamma}_J\boldsymbol\varphi^{(k)}_J,\boldsymbol\varphi^{(k')}_J\rangle_J \right)_{1\leq k,k'\leq m}$ where $\boldsymbol\varphi^{(k)}_J=(\varphi^{(k)}_j, j\in J)\in\mathbf H_J$ and $\langle\mathbf f,\mathbf g\rangle_J=\sum_{j\in J}\langle f_j,g_j\rangle_j$ is the usual scalar product of $\mathbf H_J$.
 
 Since it has been proven by \citet{CJ10b} that the rate of decrease of the eigenvalues of the covariance operator influences the minimax rates in functional linear regression, we may assume that the rate of decrease of $\tilde\kappa_n^{(m)}(s)$ to 0 influences the rate, which is confirmed by our results.
 
 \subsection{Behavior of the sequence $(\tilde\kappa_n^{(m)})_{m\geq 1}$ in some examples}
 \label{sec:examples}
 
In this section, we detail three examples of spaces $\mathbf H$ on which we will illustrate the theoretical results of the paper. The two first examples are illustrative ones and the third one is close to the electricity consumption case presented in Section~\ref{sec:applis_elec}.

\paragraph{Example 1: finite-dimensional space verifying the restricted eigenvalues assumption}

We first consider, as an illustrative example, the case where $\dim(\mathbf H)=d<+\infty$. In that case, without loss of generality, we can consider that $\dim(\mathbb H_j)=\mathbb R^{d_j}$ with $d_1+\hdots+d_p=d$. Moreover, we suppose in this example, that the restricted eigenvalues assumption $(A_{RE(s)})$ written in Section~\ref{sec:REinf} holds with $c_0=3$. This case match with the model described in \citet{BT17,lounici_oracle_2011} (with, eventually, $c_0=7$ instead of $c_0=3$ in \citealt{lounici_oracle_2011}) and we can see easily that, for any nested sequence $\mathbf H^{(1)}\subset\hdots\subset\mathbf H^{(d-1)}\subset\mathbf H^{(d)}=\mathbf H$ of $\mathbf H$, 
\[
\tilde\kappa_n^{(1)}(s)\geq\hdots\tilde\kappa_n^{(d-1)}(s)\geq \tilde\kappa_n^{(d)}(s)\geq\kappa>0.
\]

\paragraph{Example 2: simple semi-functional linear model}

We suppose that $p=2$ and $\mathbb H_1=\mathbb L^2([0,1])$ and $\mathbb H_2=\mathbb R$. We consider a basis $(\widehat e_k^{(1)})_{k\geq 1}$ that diagonalizes the empirical covariance operator $\widehat\Gamma_1$ of the functional data $(X_1^1,\hdots,X_n^1)$ and we denote by $(\widehat\mu^{(1)}_k)_{k\geq 1}$ the associated non-increasing eigenvalues sequence. Remark that the orthonormal system $\{(\widehat e_k^{(1)},0), k\geq 1; (0,1)\}$ is a basis of $\mathbf H$ that diagonalizes the operator $\widehat{\boldsymbol\Gamma}$.  We construct for a rank $r\in\mathbb N\backslash\{0\}$, 
\begin{align*}
\mathbf H^{(m)}&=\text{span}\{(\widehat e_k^{(1)},0), k=1,\hdots,m\} &\text{ for }m<r,\\
\mathbf H^{(m)}&=\text{span}\{(\widehat e_k^{(1)},0), k=1,\hdots,m-1; (0,1)\} &\text{ for }m\geq r.
\end{align*}
For $s=1$, we remark that, for $m<r$, $$\tilde\kappa_n^{(m)}(1)=\widehat\mu^{(1)}_m$$ the $m$-largest eigenvalue of $\widehat\Gamma_1$. For $m\geq r$, we also take into account the interaction between the two variables and we can see that $\tilde\kappa_n^{(m)}(1)$ is the smallest eigenvalue of the covariance matrix of the data matrix containing the coefficients of the projection of the data onto $\mathbf H^{(m)}$ which is $(\langle X_i^1,e_1^{(1)}\rangle_1,\hdots,\langle X_i^1,e_{m-1}^{(1)}\rangle_1,X_i^2)_{i=1,\hdots,n}$.

\paragraph{Example 3: fully multivariate functional linear model} 
Now we consider the example of $p$ an integer and $\mathbb H_j=\mathbb L^2([0,1])$. We define, for all $j=1,\hdots,p$, a basis $(\widehat e_k^{(j)})_{k\geq 1}$ that diagonalizes the empirical covariance operator $\widehat\Gamma_j$ of the data $(X_1^j,\hdots,X_n^j)$ and we denote by $(\widehat\mu^{(j)}_k)_{k\geq 1}$ the associated eigenvalues sequence. To simplify the definitions, we set $m=Lp$, with $L\in\mathbb N\backslash\{0\}$ and writes
\[
\mathbf H^{(m)}=S_L^{(1)}\times\hdots\times S_L^{(p)} \text{ with } S_L^{(j)}=\text{span}\{\widehat e_k^{(j)}, k=1,\hdots,L\}, j=1,\hdots,p.
\]
In that case, the matrix $\widehat{\boldsymbol\Gamma}_{|m}$ is a block matrix 
$$
\widehat{\boldsymbol\Gamma}_{|m}=\begin{pmatrix}
\widehat\Gamma_{1,1}^L&
\widehat\Gamma_{1,2}^{L} &\hdots & \widehat\Gamma_{1,p}^{L}\\

\widehat\Gamma_{1,2}^{L} & 
\widehat\Gamma_{2,2}^L
&
\hdots & \widehat\Gamma_{2,p}^{L}\\
\vdots   && \ddots\\
\widehat\Gamma_{1,p}^{L} &\hdots&  \widehat\Gamma_{p-1,p}^{L} & 
\widehat\Gamma_{p,p}^L
\end{pmatrix} 
$$
where $\widehat\Gamma_{j,j'}^{L}=\left(\frac1n\sum_{i=1}^n\langle X_i^j,\widehat e_k^{(j)}\rangle \langle X_i^{j'},\widehat e_{j'}^{(k')}\rangle\right)_{k,k'=1,\hdots,L}$ is the correlation matrix between the projections of $X^j$ into $S_j^{(L)}$ and $X^{j'}$ into $S_{j'}^{(L)}$. We remark that the matrices $\widehat\Gamma_{j,j}^L$ are diagonal matrices, with diagonal coefficients $\{\widehat\mu_k^{(j)}\}_{k=1,\hdots,L}$. Equation~\eqref{eq:encadrement_tildekappa} can be rewritten in that case
\[
\min_{j=1,\hdots,p}\widehat\mu_j^{(L)}\geq \tilde\kappa_n^{(m)}(1)\geq \tilde\kappa_n^{(m)}(2)\geq\hdots\geq\tilde\kappa_n^{(m)}(p)=\rho\left(\widehat{\boldsymbol\Gamma}_{m}^{-1/2}\right)^{-1},
\]
with equality in the case where the extra-diagonal correlation matrices $\widehat\Gamma_{j,j'}^{L}=0$ for all $j\neq j'$.  
\section{Sharp sparsity oracle-inequalities for the empirical prediction error}
\label{sec:ineg_oracle_empirical}

In this section, the design $\mathbf X_1,\hdots,\mathbf X_n$ is supposed to be either fixed or random. The results of the section are obtained under the unique assumption of Gaussianity of the noise and no assumption on the design. We prove the following sharp sparsity-oracle inequality for the solutions of both problems \eqref{eq:LASSOinf} and \eqref{eq:LASSOproj}. 
\begin{proposition} \label{prop:oracle}

Let $q>0$ be fixed and choose
\begin{equation}\label{eq:deflambda}
\lambda_j=r _n\left(\frac1n\sum_{i=1}^n\|X_i^j\|_j^2\right)^{1/2}\text{ with } r_n= A\sigma\sqrt\frac{q\ln(p)}{n} \quad(A\geq 4\sqrt 2).
\end{equation}
With probability larger than $1-p^{1-q}$, for all $m\geq 1$,
\begin{equation}\label{eq:oracle_proj}
\left\|\widehat{\boldsymbol\beta}_{\boldsymbol\lambda,m}-\boldsymbol\beta^*\right\|_n^2\leq \min_{\boldsymbol\beta\in\mathbf H^{(m)}, |J(\boldsymbol\beta)|\leq s}\left\{\left\|\boldsymbol\beta-\boldsymbol\beta^*\right\|_n^2+\frac{9}{4(\tilde\kappa_n^{(m)})^2}\sum_{j\in J(\boldsymbol\beta)}\lambda_j^2\right\}
\end{equation}
and 
\begin{equation}\label{eq:oracle}
\n{\widehat{\boldsymbol\beta}_{\boldsymbol\lambda,\infty}-\boldsymbol\beta^*}_{n}^2\leq \min_{m\geq 1}\min_{\boldsymbol\beta\in\mathbf H^{(m)}, |J(\boldsymbol\beta)|\leq s}\left\{\n{\boldsymbol\beta-\boldsymbol\beta^*}_{n}^2+\frac{9}{4(\tilde\kappa_n^{(m)})^2}\sum_{j\in J(\boldsymbol\beta)}\lambda_j^2+R_{n,m}\right\},\end{equation}
with
\[
R_{n,m}:= \sqrt{\sum_{j\in J(\boldsymbol\beta)}\lambda_j^2}\left(\n{\estim^{(\perp m)}}+\frac3{\kappa_n^{(m)}}\n{\estim^{(\perp m)}}_n\right),
\]
where $\widehat{\boldsymbol\beta}^{(\perp m)}=\widehat{\boldsymbol\beta}-\widehat{\boldsymbol\beta}^{(m)}$ the orthogonal projection onto $(\mathbf H^{(m)})^{\perp}$ and using the convention $1/0=+\infty$ in the case where $\tilde\kappa_n^{(m)}=0$.

 \end{proposition}

The proof of this result can be found in Section~\ref{subsec:proof:prop:oracle}. It is based on the ones of \citet[Proposition 5]{BT17} and \citet[Theorem 3.1]{lounici_oracle_2011} with some adjustments linked with the infinite-dimensional nature of the data. In particular, we need a concentration inequality that remains true in Hilbert spaces (see Proposition~\ref{prop:LT91}).

In the case where $\dim(\mathbf H)<+\infty$ (Example 1 in Section~\ref{sec:examples}) , we remark that when $m=d=\dim(\mathbf H)$ the problematic remaining term $R_{m,n}$ disappears and the result of Proposition~\ref{prop:oracle} coïncides with 
\begin{itemize}
\item the result of \citet[Proposition 5]{BT17} in the case $\lambda_j=\lambda$ for all $j=1,\hdots,p$ with the same constants,
\item the result of \citet[Theorem 3.1]{lounici_oracle_2011} with better constants ($9/4$ instead of $96$ in the term due to the penalty and $1$ replaced by $2$ in the bias term).
\end{itemize}

However, in the case where $\dim(\mathbf H)=+\infty$ we have to deal either with the remaining term $R_{n,m}$ for $\widehat{\boldsymbol\beta}_{\boldsymbol\lambda,\infty}$ or with the choice of an optimal dimension $m$ for $\widehat{\boldsymbol\beta}_{\boldsymbol\lambda,m}$. Up to now, it seems difficult to know the exact convergence rate of $R_{n,m}$.  On the contrary, the choice of dimension $m$ for the estimator $\widehat{\boldsymbol\beta}_{\boldsymbol\lambda,m}$ is linked with a classical bias-variance compromise. 
\begin{itemize}
\item When $m$ is small the distance $\|\boldsymbol\beta^*-\boldsymbol\beta\|_n$ between $\boldsymbol\beta^*$ and any $\boldsymbol\beta\in\mathbf H^{(m)}$ is generally large.
\item When $m$ is sufficiently large, we know the distance $\|\boldsymbol\beta^*-\boldsymbol\beta\|_n$ is small but the term $\frac{3}{(\kappa_n^{(m)})^2}\sum_{j\in J(\boldsymbol\beta)}\lambda_j^2$ may be very large since $\kappa_n^{(m)}$ is close to 0 when $m$ is close to ${\rm rk}(\widehat{\boldsymbol\Gamma})$.
\end{itemize} 

To achieve the best trade-off between these two terms, a model selection procedure, in the spirit of~\citet{BBM99}, is introduced. We select
\begin{equation}\label{eq:dim_selec}
\widehat m\in{\arg\min}_{m=1,...,N_n}\left\{\frac1n\sum_{i=1}^n\left(Y_i-\langle\widehat{\boldsymbol\beta}_{\boldsymbol\lambda,m},\mathbf X_i\rangle\right)^2+\kappa\sigma^2\frac{m\log(n)}{n}\right\},
\end{equation}
where $\kappa>0$ is a constant which can be calibrated by a simulation study or selected from the data by methods stemmed from slope heuristics (see e.g. \citealt{BMM12}) and $N_n\leq n$. 

We obtain the following sparsity oracle inequality for the selected estimator $\widehat{\boldsymbol\beta}_{\boldsymbol\lambda,\widehat m}$.%
\begin{theorem}\label{thm:oracle} Let $q>0$ and $\boldsymbol\lambda=(\lambda_1,\hdots,\lambda_p)$ chosen as in Equation~\eqref{eq:deflambda}. There exist a minimal value $\kappa_{\min}$ and a universal constant $C_{MS}>0$ such that, with probability larger than $1-p^{1-q}-C_{MS}/n$, if $\kappa>\kappa_{\min}$, for all $\tilde\eta>0$,
\begin{multline*}
\left\|\widehat{\boldsymbol\beta}_{\boldsymbol\lambda,\widehat m}-\boldsymbol\beta^*\right\|_n^2\leq  (1+\tilde\eta)\min_{m=1,\hdots,N_n}\min_{\boldsymbol\beta\in\mathbf H^{(m)},|J(\boldsymbol\beta)|\leq s}\left\{\left\|\boldsymbol\beta-\boldsymbol\beta^*\right\|_n^2+\frac{9}{4(\tilde\kappa_n^{(m)}(s))^2}\sum_{j\in J(\boldsymbol\beta)}\lambda_j^2\right.\\\left.+C(\tilde\eta)\kappa\log(n)\sigma^2\frac{m}{n}\right\},
\end{multline*}
with $C(\tilde\eta)=(\tilde\eta+2)/(\tilde\eta+1)$, $\kappa$ is the penalty constant appearing in Eq.~\eqref{eq:dim_selec} and $\tilde\kappa_n^{(m)}$ is the restricted eigenvalue quantity of Eq.~\eqref{eq:tildekappa}.
\end{theorem}

The proof of Theorem~\ref{thm:oracle} can be found in Section~\ref{subsec:proof:thm:oracle}. It is based on the control of an empirical process naturally associated with our problem given in Lemma~\ref{lem:control_nun_empir}. Both quantities $\kappa_{\min}$ and $C_{MS}$ are universal constants. 
 
Theorem~\ref{prop:oracle} implies that, with probability larger than $1-p^{1-q}-C_{MS}/n$, if $|J(\boldsymbol\beta^*)|\leq s$,
\begin{equation}\label{eq:oracle_sparse}
\begin{split}
\n{\widehat{\boldsymbol\beta}_{\boldsymbol\lambda,\widehat m}-\boldsymbol\beta^*}_{n}^2&\leq (1+\tilde\eta)\min_{m=1,\hdots,\min\{N_n,M_n\}}\left\{\n{\boldsymbol\beta^{(*,\perp m)}}_n^2+\frac{9}{4(\tilde\kappa_n^{(m)}(s))^2}\sum_{j\in J(\boldsymbol\beta^*)}\lambda_j^2\right.\\
&\qquad\left.+C(\tilde\eta)\kappa\log(n)\frac{m}{n}\right\},
\end{split}
\end{equation}
where, for all $m$, $\boldsymbol\beta^{(*,\perp m)}$ is the orthogonal projection of $\boldsymbol\beta^*$ onto $(\mathbf H^{(m)})^\perp$. The upper-bound in Equation~\eqref{eq:oracle_sparse} is then the best compromise between two terms:
\begin{itemize}
\item an approximation term $\n{\boldsymbol\beta^{(*,\perp m)}}_n^2$ which decreases to 0 when $m\to+\infty$;
\item a second term due to the penalization and the projection which increases to $+\infty$ when $m\to+\infty$. 
\end{itemize}

\section{Oracle-inequality for prediction error}
\label{sec:ineg_oracle_prediction}

The aim of this section is to prove sparsity-oracle inequalities for the theoretical counterpart of the empirical prediction error and to derive convergence rates under appropriate regularity assumptions. 

We suppose in this section that the design $\mathbf X_1,\hdots,\mathbf X_n$ is a sequence of i.i.d centered random variables in $\mathbf H$. The aim is to control the estimator in terms of the norm associated to the prediction error of an estimator $\widehat{\boldsymbol\beta}$ defined by
\[
\|\boldsymbol\beta^*-\widehat{\boldsymbol\beta}\|_{\boldsymbol\Gamma}^2=\mathbb E\left[\left(\mathbb E[Y|\mathbf X]-\langle\widehat{\boldsymbol\beta},\mathbf X\rangle\right)^2|(\mathbf X_1,Y_1),\hdots,(\mathbf X_n,Y_n)\right]=\langle\boldsymbol\Gamma(\boldsymbol\beta^*-\widehat{\boldsymbol\beta}),\boldsymbol\beta^*-\widehat{\boldsymbol\beta}\rangle.
\]
where $(\mathbf X,Y)$ follows the same distribution as $(\mathbf X_1,Y_1)$ and is independent of the sample. 


\subsection{Moment assumptions and definitions}

First denote by $\boldsymbol\Gamma:\mathbf f \in\mathbf H\mapsto \mathbb E\left[\langle \mathbf f,\mathbf X_1\rangle\mathbf X_1\right]$ the theoretical covariance operator and define a theoretical version of $\tilde\kappa_n^{(m)}$, 
 \begin{eqnarray*}
 \kappa^{(m)}(s)&:=&\\
 &&\hspace{-2cm}\small\min\left\{\frac{\n{\boldsymbol \delta}_{\boldsymbol\Gamma}}{\sqrt{\sum_{j\in J}\n{\delta_j}_j^2}}, |J|\leq s, \boldsymbol\delta=(\delta_1,...,\delta_p)\in\mathbf H^{(m)}\backslash\{0\}, \sum_{j\notin J}\lambda_j\n{\delta_j}_j\leq 3\sum_{j\in J}\lambda_j\n{\delta_j}_j\right\}. 
 \end{eqnarray*}
we also denote by $(\mu_k)_{k\geq 1}$ the eigenvalues of $\boldsymbol\Gamma$ sorted in decreasing order.

\bigskip


$(H_{Mom}^{(1)})$ There exists a constant $b>0$ such that, for all $\ell\geq 1$, 
\[
\sup_{j\geq 1}\mathbb E\left[\frac{\langle \mathbf X,\varphi^{(j)}\rangle^{2\ell}}{\tilde v_j^{\ell}}\right]\leq \ell!b^{\ell-1}\text{ where }\tilde v_j:={\rm Var}(\langle\mathbf X_i,\boldsymbol\varphi^{(j)}\rangle).
\]

($H_{Mom}^{(2)}$) There exist two constants $v_{Mom}>0$ and $c_{Mom}>0$, such that, for all $\ell\geq 2$, 
\[
\mathbb E\left[\|\mathbf X\|^{2\ell}\right]\leq \frac{\ell !}2 v_{Mom}^2 c_{Mom}^{\ell-2}. 
\]

Both assumptions $(H_{Mom}^{(1)})$ and $(H_{Mom}^{(2)})$ are necessary to apply exponential inequalities and are verified e.g. by Gaussian or bounded processes. 

\subsection{Sparsity oracle inequality}

\begin{theorem}\label{thm:oracle_pred}
Suppose  that both $(H_{Mom}^{(1)})$ and $(H_{Mom}^{(2)}$) are verified. Suppose also that $q>0$ and $\boldsymbol \lambda=(\lambda_1,\hdots,\lambda_p)$ verify the conditions of Equation~\eqref{eq:deflambda}. 
  

Then, there exist $C_{MS},c_{\max}>0$ (depending only on ${\rm tr}(\boldsymbol\Gamma)$ and $\rho(\boldsymbol\Gamma)$) and $C_{Mom}>0$ (depending only on $v_{Mom}$ and $c_{Mom}$) such that the following inequality holds with probability larger than $1-p^{1-q}-(C_{MS}+C_{Mom})/n+2n^2\exp(-c_{\max}n)$, 
  \begin{eqnarray*}
\left\|\widehat{\boldsymbol\beta}_{\boldsymbol\lambda,\widehat m}-\boldsymbol\beta^*\right\|_{\Gamma}^2&\leq &C'\min_{m=1,\hdots,N_n}\min_{\boldsymbol\beta\in\mathbf H^{(m)}}\left\{\|\boldsymbol\beta-\boldsymbol\beta^*\|_{\boldsymbol\Gamma}^2+\frac1{\left(\kappa_n^{(m)}(s)\right)^2}\left(\sum_{j\in J(\boldsymbol\beta)}\lambda_j^2+\frac{\log^2(n)}{n}\right)\right.\\
&&\left.+\kappa\frac{\log n}n\sigma^2m+\left\|\boldsymbol\beta^*-\boldsymbol\beta^{(*,m)}\right\|_{\boldsymbol\Gamma}^2+\left(\kappa_n^{(m)}(s)\right)^2\left\|\boldsymbol\beta^*-\boldsymbol\beta^{(*,m)}\right\|^2\right\}.
  \end{eqnarray*}
where $C>0$ is a universal constant.
\end{theorem}

The proof is based on concentration inequalities of ratio of norms that can be found in Proposition~\ref{prop:norm_equiv} and that relies mainly on Bernstein's inequality (for real and functional random variables). It can be found in Section~\ref{proof:thm:oracle_pred}
\subsection{Convergence rates}

From Theorem~\ref{thm:oracle_pred}, we derive an upper-bound on the convergence rates of the estimator $\widehat{\boldsymbol\beta}_{\boldsymbol\lambda,\widehat m}$. For this we need some regularity assumptions on $\boldsymbol\beta^*$ and $\boldsymbol\Gamma$. 

For a sequence $v=(v_j)_{j\geq 1}$ of positive real numbers, we define a weighted norm as follows
\[
\|\mathbf f\|_v^2:=\sum_{k\geq 1}v_k\langle\mathbf f,\boldsymbol\varphi^{(k)}\rangle^2, \qquad \mathbf f\in\mathbf H.
\]

We introduce two sequences $\mathbf{\frak  b}=(\frak b_k)_{k\geq 1}$ and $v=(v_k)_{k\geq 1}$ of positive real numbers and $R,c>0$ and note 
\[
\mathcal E_{\frak b}(R):=\left\{\boldsymbol\beta\in\mathbf H, \|\boldsymbol\beta\|_{\frak b}\leq R\right\},
\]
and 
\begin{equation*}
\mathcal N_v(c):=\{T\in\mathcal L(\mathbf H), \|\boldsymbol{\Gamma}^{1/2}\mathbf f\|\leq c\|\mathbf f\|_v, \text{ for all }f\in\mathbf H\},
\end{equation*}
for the regularity classes of $\boldsymbol\beta^*$ and $\boldsymbol\Gamma$.

Let us explain the regularity assumptions on $\bbeta$ and $\boldsymbol{\Gamma}$ in the three examples of section~\ref{sec:examples}. In order to simplify the presentation, we replace in examples 2 and 3, for all $j$, the basis $(\widehat e_k^{(j)})_{k\geq 1}$ by its theoretical counterpart $(e_k^{(j)})_{k\geq 1}$ which is the basis that diagonalizes $\Gamma_j$ and write it example 2' (resp. example 3') instead of example 2 (resp. example 3). 

In example 1, since $\dim(\mathbf H)<+\infty$, we can remark that, for any sequence $\mathfrak b\in(\mathbb R_+^*)^{\mathbb N\backslash\{0\}}$,
\[
\mathbf H=\bigcup_{R>0}\mathcal E_{\mathfrak b}(R).
\]
Similarly, it is easily seen that for all sequence $v\in(\mathbb R_+^*)^{\mathbb N\backslash\{0\}}$, there exists $c=\rho(\boldsymbol\Gamma^{1/2})/\min_{j}\{v_j^{1/2}\}>0$ such that $\boldsymbol\Gamma\in\mathcal N_v(c)$. 

In example 2', remark that, for all $\mathbf f=(f_1,f_2)\in\mathbf H=\mathbb L^2([0,1])\times\mathbb R$, 
\[
\|\mathbf f\|_v^2=v_rf_2^2+\sum_{k\neq r}v_k\langle f_1,\widehat e_k^{(1)}\rangle^2. 
\]
Then, for all $\boldsymbol\beta=(\beta_1,\beta_2)\in\mathcal E_b(R)$, there exists $R'>0$ such that $\sum_{k\neq r}v_k\langle \boldsymbol\beta,\widehat e_k^{(1)}\rangle^2\leq R'$ (and conversely). The assumption is therefore an ellipsoidal regularity assumption on the functional element of the vector $\boldsymbol\beta$. This ellispoïdal regularity assumption is very classical in non-parametric minimax estimation \citep{T09} and in particular in a functional data framework \citep{CJ10b,CJ12,BMR16}. Concerning the hypothesis on $\boldsymbol\Gamma$, we have the following characterization : there exists $c>0$ such that $\boldsymbol\Gamma\in\mathcal N_v(c)$ if and only if $\mu_k^{(1)}\lesssim  v_k$ where $(\mu_k^{(1)})_{k\geq 1}$ is the sequence of eigenvalues of the covariance operator $\Gamma_1$, sorted in non increasing order. 

Concerning the more complex example 3', there is a link between the regularity of beta $\boldsymbol\beta$ and the regularity of its coordinates. Remark that, defining $\boldsymbol\varphi^{(jp+k)}=(e_k^{(j)}\mathbf 1_{\ell=j})_{\ell=1,\hdots,p}$ we have, for all $\boldsymbol\beta=(\beta_1,\hdots,\beta_p)\in\mathbf H$, 
\[
\|\boldsymbol\beta\|_{\mathfrak b}^2=\sum_{j=1}^p\sum_{k\geq 1}\mathfrak b_{pj+k}\langle \boldsymbol\beta,\boldsymbol\varphi^{(pj+k)}\rangle^2=\sum_{j=1}^p\sum_{k\geq 1}\mathfrak b_{pj+k}\langle \beta_j,e_k^{(j)}\rangle_j^2. 
\]
Then, if each coordinate $\beta_j$ is in an ellipsoïd of $\mathbb L^2([0,1])$ i.e. there exist $b_1,\hdots,b_p>0$ and $R>0$ such that, 
\[
\sum_{k\geq 1}k^{b_j}\langle\beta_j,e_k^{(j)}\rangle_j^2\leq R,
\]
then, denoting, $\mathfrak b=(k^{\max_{j=1,\hdots,p}\{b_j\}})_{k\geq 1}$ we have 
\[
\boldsymbol\beta\in\mathcal E_{\mathfrak b}(pR).
\]
Then, the index $b_j$ accounting for the regularity of the function $\beta_j$, the vector of functions $\boldsymbol\beta$ has the worst regularity of all its coordinates. Regarding the regularity class $\mathcal N_v(c)$ a similar result may be obtained. We can see that, for all $\mathbf f=(f_1,\hdots,f_p)\in\mathbf H$,
 \[
 \|\boldsymbol\Gamma^{1/2}\mathbf f\|^2={\rm Var}(\langle\mathbf X,\mathbf f\rangle)={\rm Var}\left(\sum_{j=1}^p\langle X^j,f_j\rangle_j\right)\leq\sum_{j=1}^p{\rm Var}\left(\langle X^j,f_j\rangle_j\right)=\sum_{j=1}^p\|\Gamma_j^{1/2}f_j\|_j^2,
 \]
 and finally, if there exists $c>0$ such that $\mu_k^{(j)}\leq c v_{jp+k}$
 \[
  \|\boldsymbol\Gamma^{1/2}\mathbf f\|^2=\sum_{j=1}^p\sum_{k\geq 1}\mu_k^{(j)}\langle f_j,e_k^{(j)}\rangle_j^2\leq c\sum_{k\geq 1}v_k\langle\mathbf f,\boldsymbol\varphi^{(k)}\rangle^2
 \]
meaning that $\boldsymbol\Gamma\in\mathcal N_{v}(pc) $.

\begin{corollary}[Rates of convergence]\label{cor:rates}

We suppose that all assumptions of Theorem~\ref{thm:oracle_pred} are verified and we choose, for all $j=1,\hdots,p$, 
\[
\lambda_j=A\sigma \sqrt{\frac{\ln(n)+\ln(p)}n}\sqrt{\frac1n\sum_{i=1}^n\|X_i^j\|_j^2}, 
\]
with $A>0$ a numerical constant.

We also suppose that there exist $\gamma\geq1/2$ and $b>0$, such that 
\[
v_k = k^{-2\gamma} \text{ and }\frak b_k\asymp k^{2b}.
\]
and that there exists $\gamma(s)\geq 1/2$ such that
\[
\kappa^{(m)}(s)\asymp m^{-2\gamma(s)}.
\]
Then, there exist two quantities $C,C'>0$, such that, if $|J(\boldsymbol\beta^*)|\leq s$, with probability larger than $1-C/n$, 
\begin{equation}\label{eq:rates}
\sup_{\boldsymbol\beta^*\in\mathcal E_{\frak b}(R),\boldsymbol\Gamma\in\mathcal N_v(c)}\left\|\widehat{\boldsymbol\beta}_{\boldsymbol\lambda,\widehat m}-\boldsymbol\beta^*\right\|_{\Gamma}^2\leq C'\left(\frac{s(\ln(p)+\ln(n))+\ln^2(n)}{n}\right)^{\frac{b+\gamma}{b+\gamma(s)+\gamma}}. 
\end{equation}

\end{corollary}

The proof relies on the results of Theorem~\ref{thm:oracle_pred}.
 
The polynomial decrease of the eigenvalues $(\mu_k)_{k\geq 1}$ of the operator $\boldsymbol\Gamma$ is also a usual assumption. The Brownian bridge and the Brownian motion on $\mathbf H=\mathbb H_1=\mathbb L^2([0,1])$ verify it with $\gamma=1$.

Remark that the rate of convergence of the selected estimator $\widehat{\boldsymbol\beta}_{\boldsymbol\lambda,\widehat m}$ is the same as the one of $\widehat{\boldsymbol\beta}_{\boldsymbol\lambda,m^*}$ where
\[
m^*\sim \left(\frac{n}{s(\ln(n)+\ln(p))+\ln^2(n)}\right)^{\frac1{2b+2\gamma(s)+2\gamma}}
\]
has the order of the optimal value of $m$ in the upper-bound of Equation~\eqref{eq:oracle_proj}.

%

We do not know however the exact order of the minimax rate when the solution $\boldsymbol\beta^*$ is sparse and if it can be achieved by either $\widehat{\boldsymbol\beta}_{\boldsymbol\lambda,m}$ or $\widehat{\boldsymbol\beta}_{\boldsymbol\lambda,\infty}$.

\section{Computing the Lasso estimator}
\label{sec:algo}

The purpose of this section is to explain how the estimators $\estim$ and $\widehat{\boldsymbol\beta}_{\boldsymbol\lambda,\widehat m}$ are computed. We first describe an algorithm which allows to obtain an approximation of $\estim$ by adapting to infinite dimension an existing finite dimensional algorithm. We then explain, in subsection~\ref{subsec:choicer}, how we choose the parameter $\boldsymbol\lambda$ (both for $\estim$ and $\widehat{\boldsymbol\beta}_{\boldsymbol\lambda,\widehat m}$) and in subsection~\ref{subsec:construction_basis} how a projection space $\mathbf H^{(m)}$ can be chosen to construct the estimator $\widehat{\boldsymbol\beta}_{\boldsymbol\lambda,\widehat m}$. Finally, we define a method to reduce the usual bias of Lasso type estimators in subsection~\ref{subsec:tikho}.

\subsection{Computational algorithm}
We propose the following algorithm to compute an approximation of  $\estim$. It can also be adapted to obtain an approximation of $\widehat{\boldsymbol\beta}_{\boldsymbol\lambda,m}$, even if, for this estimator, the usual algorithms of vanilla group-Lasso can be used directly. 

The idea is to update sequentially each coordinate $\boldsymbol\beta_1,...,\boldsymbol\beta_p$ in the spirit of the \emph{glmnet} algorithm \citep{glmnet} by solving 
\begin{equation}\label{eq:glmnet}
\begin{split}
\beta_j^{(k+1)}\in{\arg\min}_{\beta_j\in\mathbb H_j}\left\{\frac1n\sum_{i=1}^n\left(Y_i-\sum_{\ell=1}^{j-1}\ps{\boldsymbol\beta_\ell^{(k+1)}}{X_i^\ell}_\ell-\ps{\beta_j}{X_i^j}_j-\sum_{\ell=j+1}^p\ps{\boldsymbol\beta_\ell^{(k)}}{X_i^\ell}_\ell\right)^2\right.\\
\left.+2\lambda_j\n{\beta_j}_j\right\}.
\end{split}
\end{equation}
 However, in the Group-Lasso context, this algorithm is based on the so-called \emph{group-wise orthonormality condition}, which, translated to our context, amounts to suppose that the operators $\widehat{\Gamma}_j$ (or their restrictions $\widehat{\Gamma}_{j|m}$) are all equal to the identity. This assumption is not possible if $\dim(\mathbb H_j)=+\infty$ since $\widehat{\Gamma}_j$ is a finite-rank operator.  Without this condition, Equation~\eqref{eq:glmnet} does not admit a closed-form solution and, hence, is not calculable. We then propose a variant of the GPD (Groupwise-Majorization-Descent) algorithm, initially defined by \citet{YZ15} for Group-Lasso type optimization problems, without imposing the group-wise orthonormality condition. 
The GPD algorithm is also based on the principle of coordinate descent but the minimisation problem~\eqref{eq:glmnet} is modified in order to relax the group-wise orthonormality condition. We denote by $\widehat\bbeta^{(k)}$ the value of the parameter at the end of iteration $k$. During iteration $k+1$, we update sequentially each coordinate. Suppose that we have changed the $j-1$ first coordinates, the current value of our estimator is $(\widehat\beta_1^{(k+1)},..., \widehat\beta^{(k+1)}_{j-1},\widehat\beta^{(k)}_{j},...,\widehat\beta_p^{(k)})$. We want to update the $j$-th coefficient and, ideally, we would like to minimise the following criterion
\[
\gamma_n(\beta_j):=\frac1n\sum_{i=1}^n\left(Y_i-\sum_{\ell=1}^{j-1}\ps{\widehat\beta_\ell^{(k+1)}}{X_i^\ell}_\ell-\ps{\beta_j}{X_i^j}_j-\sum_{\ell=j+1}^p\ps{\widehat\beta_\ell^{(k)}}{X_i^\ell}_\ell\right)^2+2\lambda_j\|\beta_j\|_j^2.
\]
We have 
\begin{align*}
\begin{split}\gamma_n(\beta_j)-\gamma_n(\widehat\beta_{j}^{(k)})&=
-\frac2n\sum_{i=1}^n(Y_i-\widetilde Y_i^{j,k})\ps{\beta_j-\widehat\beta_{j}^{(k)}}{X_i^j}_j+\frac1n\sum_{i=1}^n\ps{\beta_j}{X_i^j}_j^2\\
&\qquad-\frac1n\sum_{i=1}^n\ps{\widehat\beta_{j}^{(k)}}{X_i^j}_j^2+2\lambda_j(\n{\beta_j}_j-\n{\widehat\beta_{j}^{(k)}}_j),
\end{split}
\end{align*}
with $\widetilde Y_i^{j,k}=\sum_{\ell=1}^{j-1}\ps{\widehat\beta_\ell^{(k+1)}}{X^\ell_i}_\ell + \sum_{\ell=j+1}^{p}\ps{\widehat\beta_\ell^{(k)}}{X^\ell_i}_\ell$, and
\begin{align*}
\frac1n\sum_{i=1}^n\ps{\beta_j}{X_i^j}_j^2-\frac1n\sum_{i=1}^n\ps{\widehat\beta_{j}^{(k)}}{X_i^j}_j^2&=
\ps{\widehat{\Gamma}_j\beta_j}{\beta_j}_j-\ps{\widehat{\Gamma}_j\widehat\beta_{j}^{(k)}}{\widehat\beta_{j}^{(k)}}_j\\
&=\ps{\widehat{\Gamma}_j(\beta_j-\widehat\beta_{j}^{(k)})}{\beta_j-\widehat\beta_{j}^{(k)}}_j+2\ps{\widehat{\Gamma}_j\widehat\beta_{j}^{(k)}}{\beta_j-\widehat\beta_{j}^{(k)}}_j.
\end{align*}
Hence 
\[
\gamma_n(\beta_j)=\gamma_n(\widehat\beta_{j}^{(k)})-2\ps{R_j}{\beta_j-\widehat\beta_{j}^{(k)}}_j+\ps{\widehat{\Gamma}_j(\beta_j-\widehat\beta_{j}^{(k)})}{\beta_j-\widehat\beta_{j}^{(k)}}_j+2\lambda_j(\n{\beta_j}_j-\n{\widehat\beta_{j}^{(k)}}_j)
\]
with 
\[
R_j=\frac1n\sum_{i=1}^n(Y_i-\widetilde Y_i^{j,k})X_i^j+\widehat{\Gamma}_j\widehat\beta_{j}^{(k)}=\frac1n\sum_{i=1}^n(Y_i-\widehat Y^{j,k}_i)X_i^j,
\]
where, for $i=1,...,n$, $\widehat Y_i^{j,k}=\widetilde Y_i^{j,k}+\ps{\widehat\beta_{j}^{(k)}}{X_i^j}_j=\sum_{\ell=1}^{j-1}\ps{\widehat\beta_\ell^{(k+1)}}{X^\ell_i}_\ell + \sum_{\ell=j}^{p}\ps{\widehat\beta_\ell^{(k)}}{X^\ell_i}_\ell$ is the current prediction of $Y_i$. If $\widehat{\Gamma}_j$ is not the identity, we can see that the minimisation of $\gamma_n(\beta_j)$ has no explicit solution. To circumvent the problem the idea is to upper-bound the quantity 
\[
\ps{\widehat{\Gamma}_j(\beta_j-\widehat\beta_{j}^{(k)})}{\beta_j-\widehat\beta_{j}^{(k)}}_j\leq \rho(\widehat{\Gamma}_j)\n{\beta_j-\widehat\beta_{j}^{(k)}}_j^2\leq N_j\n{\beta_j-\widehat\beta_{j}^{(k)}}_j^2,
\]
where $N_j:=\frac1n\sum_{i=1}^n\n{X_i^j}^2_j$ is an upper-bound on the spectral radius $\rho(\widehat{\Gamma}_j)$ of $\widehat{\Gamma}_j$. Instead of minimising $\gamma_n$ we minimise its upper-bound
\[
\widetilde\gamma_n(\beta_j)=-2\ps{R_j}{\beta_j}_j+N_j\n{\beta_j-\widehat\beta_{j}^{(k)}}_j^2+2\lambda_j\n{\beta_j}_j. 
\]
The minimisation problem of $\widetilde\gamma_n$ has an explicit solution 
\begin{equation}\label{eq:updates}
\widehat\beta_j^{(k+1)}=\left(\widehat\beta_j^{(k)}+\frac{R_j}{N_j}\right)\left(1-\frac{\lambda_j}{\n{N_j\widehat\beta_j^{(k)}+R_j}_j}\right)_+.
\end{equation}

After an initialisation step $(\boldsymbol\beta^{(0)}_1,...,\boldsymbol\beta^{(0)}_p)$, the updates on the estimated coefficients are then given by Equation~\eqref{eq:updates}.
 
Remark that, for the case of Equation~\eqref{eq:LASSOinf}, the optimisation is done directly in the space $\mathbf H$ and does not require the data to be projected. Consequently, it avoids the loss of information and the computational cost due to the projection of the data in a finite dimensional space, as well as, for data-driven basis such as PCA or PLS, the computational cost of the calculation of the basis itself.

 \subsection{Choice of smoothing parameters $(\lambda_j)_{j=1,...,p}$}
 \label{subsec:choicer}
Following Proposition~\ref{prop:oracle}, we choose $\lambda_j=\lambda_j(r)=r \left(\frac1n\sum_{i=1}^n\|X_i^j\|_j^2\right)^{1/2}$, for all $j=1,...,p$. This allows to restrain the problem of the calibration of the $p$ parameters $\lambda_1,...,\lambda_p$ to the calibration of only one parameter $r$. In this section, we write 	$\boldsymbol\lambda(r)=(\lambda_1(r),\hdots,\lambda_p(r))$ and $\widehat{\boldsymbol\beta}_{\boldsymbol\lambda(r),m}$ the corresponding minimiser of criterion~\eqref{eq:LASSOproj} if $m<+\infty$ or \eqref{eq:LASSOinf} if $m=+\infty$.

Drawing inspiration from~\citet{glmnet}, we consider a pathwise coordinate descent scheme starting from the following value of $r$,
$$r_{\max}=\max_{j=1,...,p}\left\{\frac{\n{\frac1n\sum_{i=1}^nY_iX_i^j}_j}{\sqrt{\frac1n\sum_{i=1}^n\n{X_i^j}_j^2}}\right\}.$$
It can be proven that, taking $r=r_{\max}$, the solution of the minimisation problem~\eqref{eq:LASSOinf} is $\widehat{\boldsymbol\beta}_{\boldsymbol\lambda(r_{\max})}=(0,...,0)$.  Starting from this value of $r_{\max}$, we choose a grid decreasing from $r_{\max}$ to $r_{\min}=\delta r_{\max}$ of $n_r$ values equally spaced in the log scale i.e. 
\begin{eqnarray*}
\mathcal R&=&\left\{\exp\left(\log(r_{\min})+(k-1)\frac{\log(r_{\max})-\log(r_{\min})}{n_r-1}\right), k=1,...,n_r\right\}\\
&=&\{r_k, k=1,...,n_r\}.
\end{eqnarray*}
For each $k\in\{1,...,n_r-1\}$, the minimisation of criterion~\eqref{eq:LASSOinf} (resp.~\eqref{eq:LASSOproj}) with $r=r_k$ is then performed using the result of the minimisation of~\eqref{eq:LASSOinf} (resp.~\eqref{eq:LASSOproj}) with $r=r_{k+1}$ as an initialisation. As pointed out by \citet{glmnet}, this scheme leads to a more stable and faster algorithm. In practice, we chose $\delta=0.001$ and $n_r=100$. However, when $r$ is too small, the algorithm does not always converge, in particular when the dimension is large or infinite. We believe that it is linked with the fact that the optimisation problem~\eqref{eq:LASSOinf} has no solution as soon as $\dim(\mathbf H)\geq {\rm rk}(\widehat{\boldsymbol\Gamma})$ and $\boldsymbol\lambda=0$.

In the case where the noise variance is known, Theorem~\ref{prop:oracle} suggests the value $$r_n=4\sqrt{2}\sigma\sqrt{p\ln(q)/n}.$$ We recall that Equation~\eqref{eq:oracle} is obtained with probability $1-p^{1-q}$. Hence, if we want a precision better than $1-\alpha$, we take $q=1-\ln(\alpha)/\ln(p)$. However, in practice, the parameter $\sigma^2$ is usually unknown. We propose three methods to choose the parameter $r$ among the grid $\mathcal R$ and compare them in the simulation study.

\subsubsection{$V$-fold cross-validation}
We split the sample $\{(Y_i,\mathbf X_i), i=1,...,n\}$ into $V$ subsamples $\{(Y^{(v)}_i,\mathbf X^{(v)}_i), i\in I_v\}$, $v=1,...,V$, where $I_v=\lfloor(v-1)n/V\rfloor+1,...,\lfloor vn/V\rfloor$, $Y^{(v)}_i=Y_{\lfloor(v-1)n/V\rfloor+i}$, $\mathbf X^{(v)}_i=\mathbf X_{\lfloor(v-1)n/V\rfloor+i}$ and, for $x\in\R$, $\lfloor x\rfloor$ denotes the largest integer smaller than $x$. 

For all $v\in V$, $i\in I_v$, $r\in\mathcal R$ let 
$$\widehat Y_{i}^{(v,r)}=\langle\widehat{\boldsymbol\beta}^{(-v)}_{\boldsymbol\lambda(r),m},\mathbf X_i\rangle$$
be the prediction made with the estimator of $\boldsymbol\beta^*$ minimising criterion \eqref{eq:LASSOinf} (or~\eqref{eq:LASSOproj}) using only the data $\left\{(\mathbf X^{(v')}_i,Y^{(v')}_i), i\in I_{v'}, v\neq v'\right\}$. 

We choose the value of $r_n$ minimising the mean of the cross-validated error: 
$$\widehat r_n^{(CV)}\in{\arg\min}_{r\in\mathcal R}\left\{\frac1n\sum_{v=1}^V\sum_{i\in I_v}\left(\widehat Y_{i}^{(v,r)}-Y_i^{(v)}\right)^2\right\}. $$

\subsubsection{Estimation of $\sigma^2$}
We propose the following estimator of $\sigma^2$: 
\[
\widehat\sigma^2 = \frac1n\sum_{i=1}^n\left(Y_i-\ps{\widehat{\boldsymbol\beta}_{\boldsymbol\lambda(\widehat r_{\min}),m}}{\mathbf X_i}\right)^2,
\]
where $\widehat r_{\min}$ is an element of $r\in\mathcal R$. 

In practice, we take the smallest element of $\mathcal R$ for which the algorithm converges.

We set
\[
\widehat r_n^{(\widehat\sigma^2)}:=4\sqrt{2}\widehat\sigma\sqrt{p\ln(q)/n}\text{ with }q=1-\ln(5\%)/\ln(p).
\]

\subsubsection{BIC criterion}
We also consider the BIC criterion, as proposed by \citet{WLT07,WL07},
\[
\widehat r_n^{(BIC)}\in{\arg\min}_{r\in\mathcal R}\left\{\log(\widehat\sigma_r^2)+|J(\widehat{\boldsymbol\beta}_{\boldsymbol\lambda(r),m})|\frac{\log(n)}n\right\}.
\]

The corresponding values of $\boldsymbol\lambda$ will be denoted respectively by $\widehat{\boldsymbol\lambda}^{(CV)}:=\boldsymbol\lambda(\widehat r_n^{(CV)})$, $\widehat{\boldsymbol\lambda}^{(\widehat\sigma^2)}:=\boldsymbol\lambda(\widehat r_n^{(\widehat\sigma^2)})$ and $\widehat{\boldsymbol\lambda}^{(BIC)}:=\boldsymbol\lambda(\widehat r_n^{(BIC)})$. The practical properties of the three methods are compared in Section~\ref{sec:simus}.

\subsection{Construction of the projected estimator}
\label{subsec:construction_basis}

The projected estimator relies mainly on the choice of the basis $(\boldsymbol\varphi^{(k)})_{k\geq 1}$. To verify the support stability condition $C_{supp}$, a possibility is to proceed as follows.
\begin{itemize}
\item Choose, for all $j=1,\hdots,p$ an orthonormal basis of $\mathbb H_j$, denoted by $(e_k^{(j)})_{1\leq j\leq\dim(\mathbb H_j)}$. 
\item Choose a bijection
\[
\boldsymbol\sigma:\begin{array}{ccl}
\mathbb N\backslash\{0\}&\to		& \{(j,k)\in\{1,\hdots,p\}\times \mathbb N\backslash\{0\}, k\leq \dim(\mathbb H_j)\}\subseteq\mathbb N^2 \\
k 		&\mapsto	& (\sigma_1(k),\sigma_2(k)). 
\end{array}
\]
\item Define 
\[
\boldsymbol\varphi^{(k)}:=(0,\hdots,0,e_{\sigma_2(k)}^{(\sigma_1(k))},0,\hdots,0)=\left(e_{\sigma_2(k)}^{(\sigma_1(k))}\mathbf 1_{\{j=\sigma_1(k)\}}\right)_{1\leq j\leq p}.
\]
\end{itemize}

There are many ways to choose the basis $(e_k^{(j)})_{1\leq k\leq\dim(\mathbb H_j)}$, $j=1,\hdots,p$ as well as the bijection $\boldsymbol\sigma$, depending on the nature of the spaces $\mathbb H_1,\hdots,\mathbb H_p$. We give here some examples.

\begin{description}
\item[Example 1: fixed basis and fixed bijection $\boldsymbol\sigma$] Suppose $\mathbb H_1=\hdots=\mathbb H_{p_\infty}=\mathbb L^2([0,1])$ and $\mathbb H_j$ are finite-dimensional for all $j=p_\infty+1,\hdots,p$.  For $j=1,\hdots,p_{\infty}$ $(e_k^{(j)})_{k\geq 1}$ is e.g. the Fourier basis 
\[
e_1^{(j)}\equiv 1, e_{2k}^{(j)}(t)=\sqrt{2}\cos(2\pi k t)\text{ and }e^{(j)}_{2k+1}(t)=\sqrt{2}\sin(2\pi k t),
\]
and, for $j=p_\infty+1,\hdots,p$, $(e_1^{(j)},\hdots,e^{(j)}_{\dim(\mathbb H_j)})$ is the canonical basis of the finite-dimensional space $\mathbb H_j$. Choosing the bijection $\boldsymbol\sigma(1)=(1,1)$, $\boldsymbol\sigma(2)=(2,1)$,...,$\boldsymbol\sigma(p)=(p,1)$, $\boldsymbol\sigma(p+1)=(1,2)$, $\boldsymbol\sigma(p+2)=(2,2)$,... leads to the basis 
\begin{eqnarray*}
\boldsymbol\varphi^{(1)}&:=&(e_1^{(1)},0,\hdots,0)\\
\boldsymbol\varphi^{(2)}&:=&(0,e_1^{(2)},0,\hdots,0)\\
&\vdots&\\
\boldsymbol\varphi^{(p)}&:=&(0,\hdots,0,e_1^{(p)})\\
\boldsymbol\varphi^{(p+1)}&:=&(e_2^{(1)},0,\hdots,0)\\
\boldsymbol\varphi^{(p+2)}&:=&(0,e_2^{(2)},0,\hdots,0)\\
&\vdots&
\end{eqnarray*}
\item[Example 2: fixed basis with random bijection $\boldsymbol\sigma$] A disadvantage of the previous example is that it gives particular importance to the first variables which is not necessarily justified by the data. A possible way to circumvent the problem is to define a random permutation $\boldsymbol\sigma$. Using the same notations as in Example 1, we can define e.g. $\boldsymbol\sigma$ as follows:
\begin{enumerate}
\item Choose $\sigma_1(1)$ uniformly in $\{1,\hdots,p\}$. 
\item If $\sigma_1(1)\leq p_\infty$, $\sigma_2(1)=1$, otherwise $\sigma_2(1)$ is chosen uniformly in $\{1,\hdots,\dim(\mathbb H_j)\}$. 
\end{enumerate}
Proceed in a similar way for $k=2,3,...$ respecting the constraint $\boldsymbol\sigma(k)\neq\boldsymbol\sigma(k')$ for $k\neq k'$.
\item[Example 3: PCA basis with data-driven choice of the bijection $\boldsymbol\sigma$]  Let, for $j=1,\hdots,p$, $(\widehat e_k^{(j)})_{1\leq k\leq\dim(\mathbb H_j)}$ the PCA basis of $\{X_i^j, i=1,\hdots,n\}$, that is to say a basis of eigenfunctions (if $\mathbb H_j$ is a function space) or eigenvectors (if $\dim(\mathbb H_j)<+\infty$) of the covariance operator $\widehat\Gamma_j$. We denote by $(\widehat\mu_k^{(j)})_{1\leq k\leq\dim(\mathbb H_j)}$ the corresponding eigenvalues. This naturally provides a data-driven choice of the bijection $\boldsymbol\sigma$ the can be defined such that $(\widehat\mu_{\sigma_2(k)}^{(\sigma_1^{(k)})})_{k\geq 1}$ is sorted in decreasing order.  
Since the elements of the PCA basis are data-dependent, but depend only on the $\mathbf X_i$'s, the results of Section~\ref{sec:ineg_oracle_empirical} hold but not the results of Section~\ref{sec:ineg_oracle_prediction}. Similar results for the PCA basis could be derived from the theory developed in~\citet{MR15,BMR16} at the price of further theoretical considerations which are out of the scope of the paper. We follow in Section~\ref{sec:simus} an approach based on the principal components basis (PCA basis). Other data-driven basis such as the Partial Leasts Squares (PLS, \citealt{preda_pls_2005,wold_soft_1975}) can also be considered in practice. 
\end{description}

 \subsection{Tikhonov regularization step}
 \label{subsec:tikho}
 It is well known that the classical Lasso estimator is biased \citep[see e.g.][Section 4.2.5]{Giraud15} because the $\ell^1$ penalization favors too strongly solutions with small $\ell^1$ norm. To remove it, one of the current method, called Gauss-Lasso, consists in fitting a least-squares estimator on the sparse regression model constructed by keeping only the coefficients which are on the support of the Lasso estimate. 
 
 This method is not directly applicable here because least-squares estimators are not well-defined in infinite-dimensional contexts. Indeed, to compute a least-squares estimator of the coefficients in the support $\widehat J$ of the Lasso estimator, we need to invert the covariance operator $\widehat{\boldsymbol{\boldsymbol\Gamma}}_{\widehat J}$ which is generally not invertible. 
 
 To circumvent the problem, we propose a ridge regression approach (also named Tikhonov regularization below) on the support of the Lasso estimate.  A similar approach has been investigated by \citet{LY13} in high-dimensional regression. They have shown the unbiasedness of the combination of Lasso and ridge regression. 
More precisely, we consider the following minimisation problem
 \begin{equation}\label{eq:ridge}
 \widetilde\beta = {\arg\min}_{\boldsymbol\beta\in\mathbf{H}_{J(\widehat{\boldsymbol\beta})}}\left\{\frac1n\sum_{i=1}^n\left(Y_i-\ps{\boldsymbol\beta}{\mathbf X_i}\right)^2+ \rho\n{\boldsymbol\beta}^2\right\}
 \end{equation}
 with $\rho>0$ a parameter which can be selected e.g. by $V$-fold cross-validation. 
 We can see that 
 $$\widetilde\beta=(\widehat{\boldsymbol{\boldsymbol\Gamma}}_{\widehat J}+\rho I)^{-1}\widehat{\boldsymbol\Delta},$$
with $\widehat{\boldsymbol\Delta}:=\frac1n\sum_{i=1}^nY_i\Pi_{\widehat J}\mathbf X_i$,
 is an exact solution of problem~\eqref{eq:ridge} but need the inversion of the operator $\widehat{\boldsymbol\Gamma}_{\widehat J}+\rho I$ to be calculated in practice. In order to compute the solution of \eqref{eq:ridge}, we define below a stochastic gradient descent algorithm. The algorithm is initialised at the solution $\widetilde{\boldsymbol\beta}^{(0)}=\widehat{\boldsymbol\beta}_{\boldsymbol\lambda(\widehat r_n^{(\widehat\sigma^2)}),m}$ (where $m=\infty$ or $m=\widehat m$) of the Lasso and, at each iteration, we do 
 \begin{equation}\label{eq:optimtikho}
 \widetilde{\boldsymbol\beta}^{(k+1)}=\widetilde{\boldsymbol\beta}^{(k)}-\alpha_k\gamma_n'(\widetilde{\boldsymbol\beta}^{(k)}),
 \end{equation}
 where 
 $$\gamma_n'(\boldsymbol\beta)=-2\boldsymbol{\widehat\Delta}+2(\widehat{\boldsymbol{\boldsymbol\Gamma}}_{\widehat J}+\rho I)\boldsymbol\beta,$$
 is the gradient of the criterion to minimise. 
 
 In practice we choose $\alpha_k=\alpha_1 k^{-1}$ with $\alpha_1$ tuned in order to get convergence at reasonable speed.

    \section{Numerical study}
    \label{sec:simus}
    
 In this section, we study practical properties of both estimators $\estim$ and $\widehat{\boldsymbol\beta}_{\boldsymbol\lambda,\widehat m}$. We first consider a context where the data are simulated and then an application to the prediction of electricity consumption.
 
\subsection{Simulation study}
We test the algorithm on two examples : 
$$Y=\ps{\boldsymbol\beta^{*,k}}{\mathbf X}+\varepsilon, k=1,2,$$
where $p=7$, $\Hb_1=\Hb_2=\Hb_3=\mathbb L^2([0,1])$ equipped with its usual scalar product $\ps{f}{g}_{\mathbb L^2([0,1])}=\int_0^1f(t)g(t)dt\text{ for all }f,g$, $\Hb_4=\R^4$ equipped with its scalar product $(a,b)=~^tab$, $\Hb_5=\Hb_6=\Hb_7=\R$, $\varepsilon\sim\mathcal N(0,\sigma^2)$ with $\sigma=0.01$. The size of the sample is fixed to $n=1000$. The definitions of $\boldsymbol\beta^{*,1}$,  $\boldsymbol\beta^{*,2}$ and $\mathbf X$ are given in Table~\ref{tab:lois}. 

\begin{table}
\footnotesize
\begin{tabular}{|c|l|l|p{0.6\textwidth}|}
\hline
      &Example 1&Example 2&\\
$j$&$\beta_j^{*,1}$&$\beta_j^{*,2}$&$X_j$\\
\hline
&&&\\
1&$t\mapsto10\cos(2\pi t)$&$t\mapsto10\cos(2\pi t)$& Brownian motion on $[0,1]$\\
&&&\\
2&0&0&$t\mapsto a + bt + c\exp(t) + \sin(dt)$ with $a\sim\mathcal U([-50,50])$, $b\sim\mathcal U([-30,30])$, $c\sim\mathcal ([-5,5])$ and $d\sim\mathcal U([-1,1])$, $a$, $b$, $c$ and $d$ independent \citep{FV00}\\
&&&\\
3&0&0&$X_2^2$\\
&&&\\
4&0&$(1,-1,0,3)^t$&$Z ~^tA$ with $Z=(Z_1,...,Z_4)$, $Z_k\sim\mathcal U([-1/2,1/2])$, $k=1,...,4$, $A=\begin{pmatrix}-1&0&1&2\\3&-1&0&1\\2&3&-1&0\\1&2&3&-1\end{pmatrix}$\\
&&&\\
5&0&0&$\mathcal N(0,1)$\\
&&&\\
6&0&0&$\|X_2\|_{\mathbb L^2([0,1])}-\mathbb E[\|X_2\|_{\mathbb L^2([0,1])}]$\\
&&&\\
7&0&1&$\|\log(|X_1|)\|_{\mathbb L^2([0,1])}-\mathbb E[\|\log(|X_1|)\|_{\mathbb L^2([0,1])}]$\\
&&&\\
\hline
\end{tabular}
\caption{\label{tab:lois}Values of $\boldsymbol\beta^{*,k}$ and $\mathbf X$}
\end{table}

\subsection{Support recovery properties and parameter selection}

\begin{figure}
\begin{tabular}{cc}
Example 1 & Example 2\\
\includegraphics[width=0.45\textwidth]{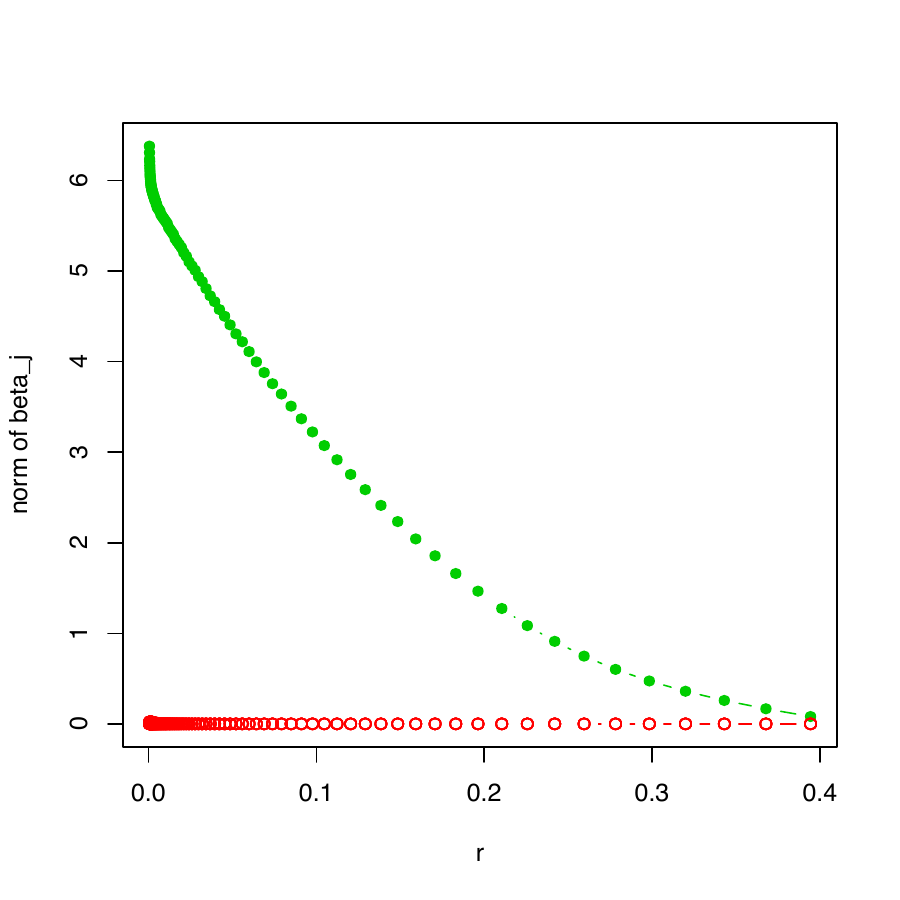}&\includegraphics[width=0.45\textwidth]{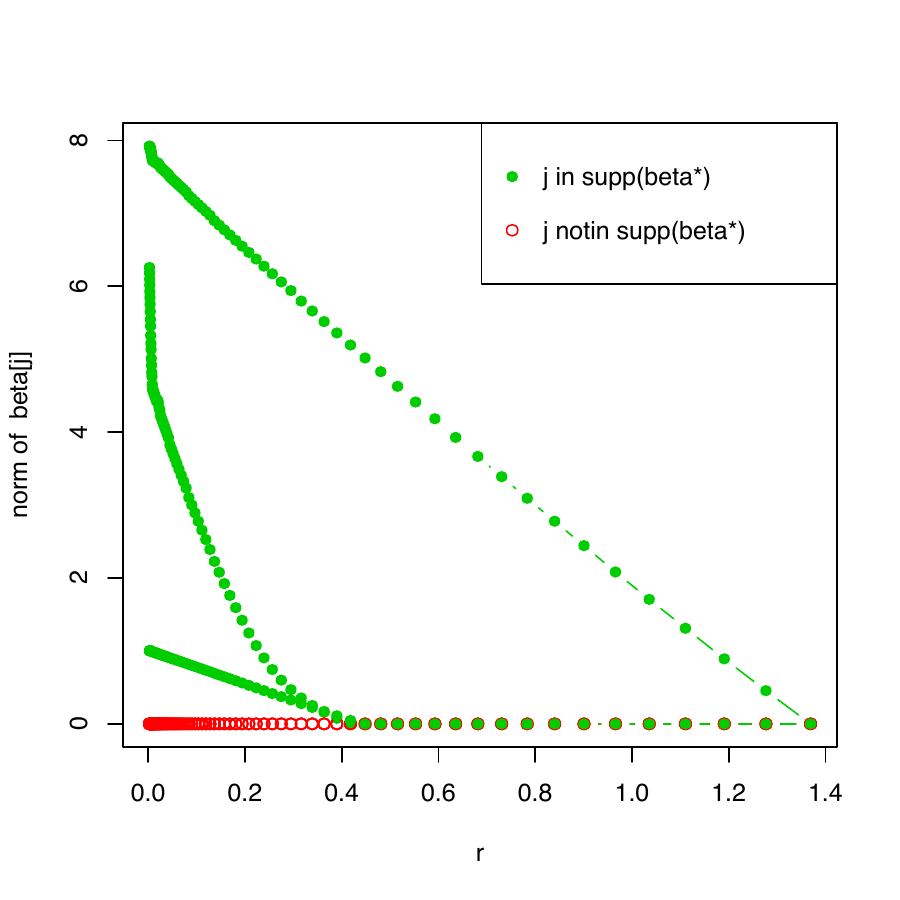}
\end{tabular}
\caption{\label{fig:normes_r} Plot of the norm of $\left[\widehat{\boldsymbol\beta}_{\boldsymbol\lambda,\infty}\right]_j$, for $j=1,...,7$ as a function of $r$.}
\end{figure}

In Figure~\ref{fig:normes_r}, we plot the norm of $\left[\widehat{\boldsymbol\beta}_{\boldsymbol\lambda,\infty}\right]_j$ as a function of the parameter $r$. We see that, for all values of $r$, we have $\widehat J\subseteq J^*$, and, if $r$ is sufficiently small $\widehat J=J^*$. We compare in Table~\ref{tab:support} the percentage of time where the true model has been recovered when the parameter $r$ is selected with the three methods described in Section~\ref{subsec:choicer}. We see that the method based on the estimation of $\widehat\sigma^2$ has very good support recovery performances, but both BIC and CV criterion do not perform well. Since the CV criterion minimises an empirical version of the prediction error, it tends to select a parameter for which the method has good predictive performances.  However, this is not necessarily associated with good support recovery properties which could explain the bad performances of the CV criterion in terms of support recovery.  As a consequence, the method based on the estimation of $\sigma^2$ is the only one which is considered for the projected estimator $\widehat{\boldsymbol\beta}_{\boldsymbol\lambda,\widehat m}$ and in the sequel we will denote simply $\widehat{\boldsymbol\lambda}=\widehat{\boldsymbol\lambda}^{(\widehat\sigma^2)}$.

\begin{table}
\begin{tabular}{|c|ccc|ccc|}
\hline
 & \multicolumn{3}{c|}{Example 1} & \multicolumn{3}{c|}{Example 2}\\
 \hline 
 & $\widehat{\boldsymbol\lambda}^{(CV)}$ & $\widehat{\boldsymbol\lambda}^{(\widehat\sigma^2)}$ & $\widehat {\boldsymbol\lambda}^{(BIC)}$& $\widehat{\boldsymbol\lambda}^{(CV)}$ & $\widehat{\boldsymbol\lambda}^{(\widehat\sigma^2)}$ & $\widehat{\boldsymbol\lambda}^{(BIC)}$\\
 \hline
 Support recovery of $\widehat{\boldsymbol{\beta}}_{\widehat{\boldsymbol\lambda},\infty}$ (\%) & 0 & 100 & 0 & 2 & 100 & 4\\
 \hline
 Support recovery of $\widehat{\boldsymbol{\beta}}_{\widehat{\boldsymbol\lambda},\widehat m}$ (\%) &$ / $& 100 & $/$& $/$& 100 & $/$\\
 \hline
\end{tabular}
\caption{\label{tab:support} Percentage of times where the true support has been recovered among 50 Monte-Carlo replications of the estimates.}
\end{table}

\subsection{Lasso estimators}
\begin{figure}
\begin{tabular}{cc}
Example 1 & Example 2\\
\includegraphics[width=0.45\textwidth,height=6cm]{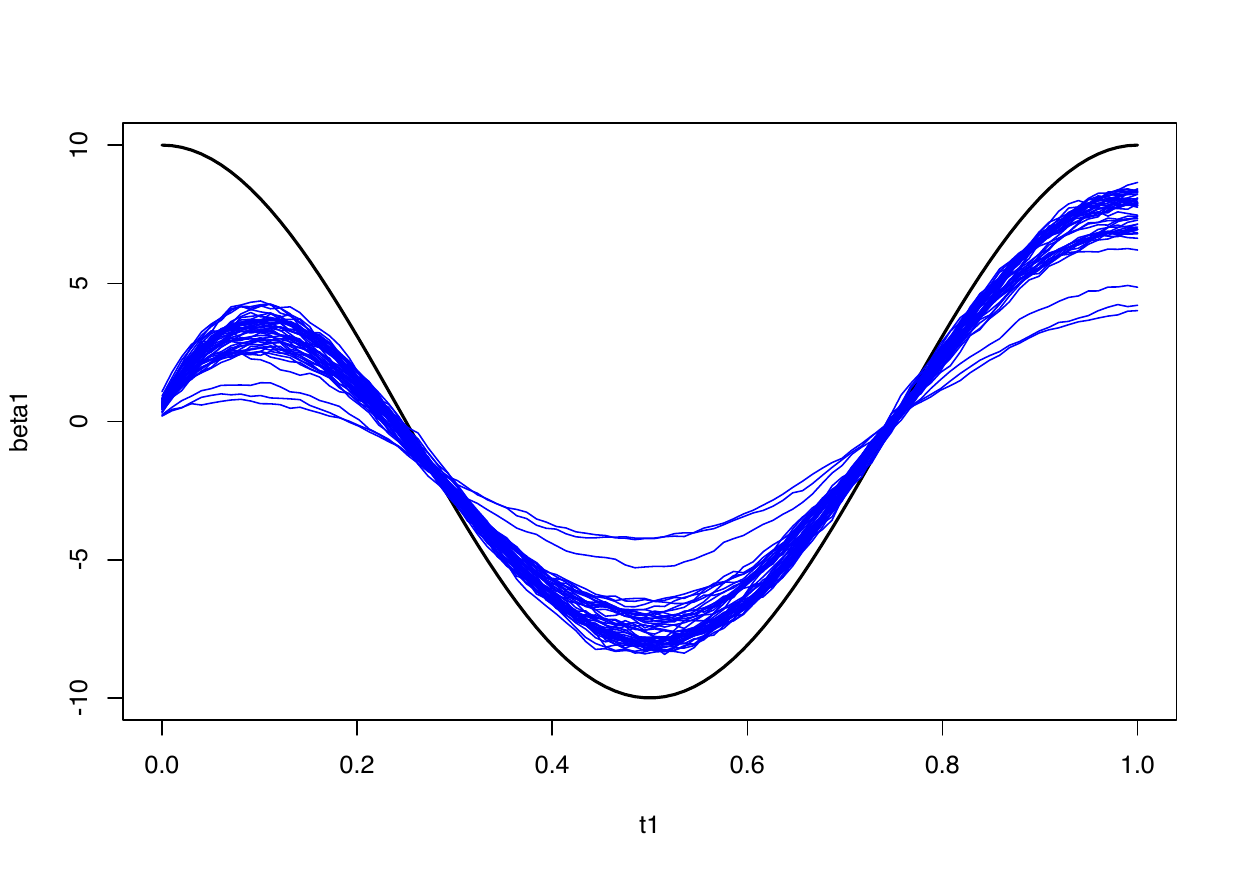}&\includegraphics[width=0.45\textwidth,height=6cm]{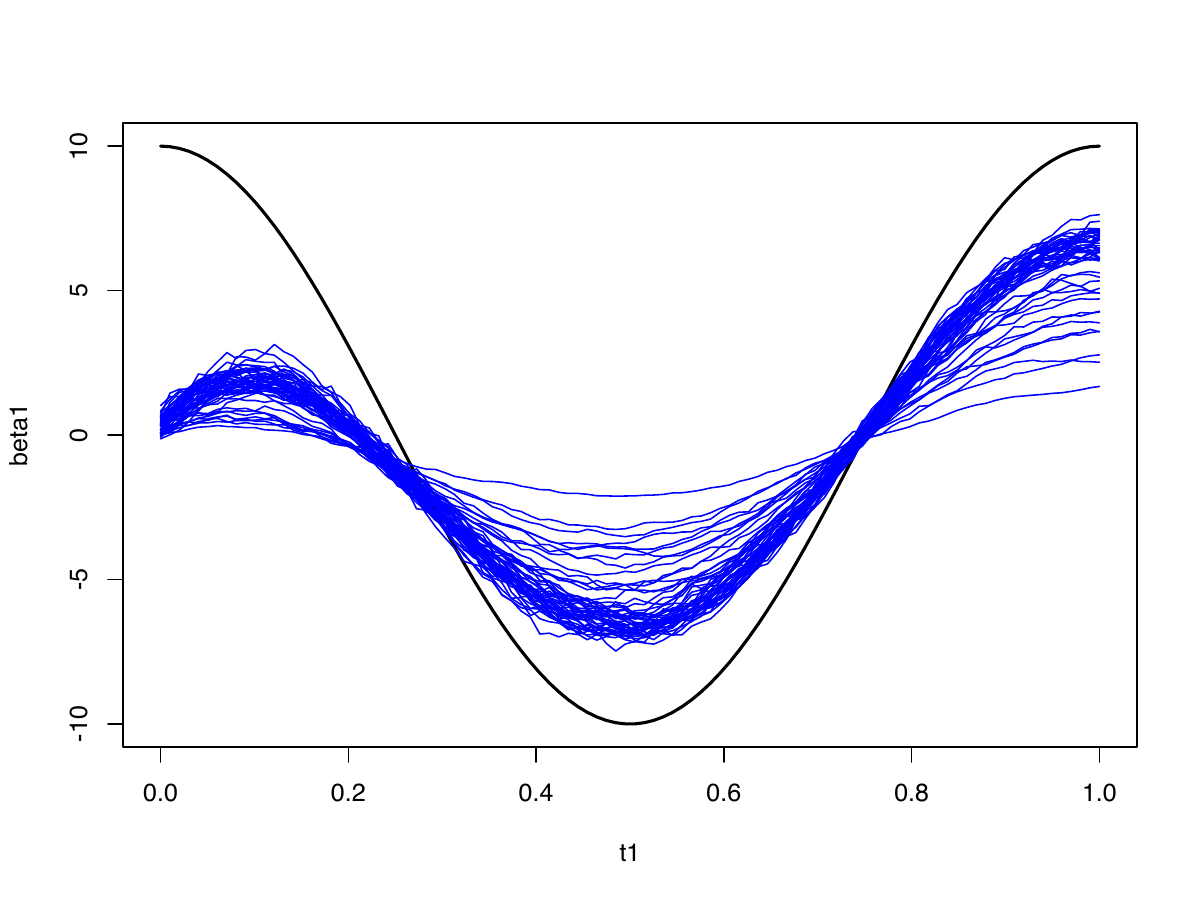}
\end{tabular}
\caption{\label{fig:beta1_Lasso_example1} Plot of $\beta_1^*$ (solid black line) and 50 Monte-Carlo replications of $\left[\widehat{\boldsymbol\beta}_{\boldsymbol\lambda,\infty}\right]_1$ (blue lines).}
\end{figure}

In Figure~\ref{fig:beta1_Lasso_example1}, we plot the first coordinate $\left[\widehat{\boldsymbol\beta}_{\widehat{\boldsymbol\lambda},\infty}\right]_1$ of Lasso estimator $\widehat{\boldsymbol\beta}_{\widehat{\boldsymbol\lambda},\infty}$ (right) and compare it with the true function $\boldsymbol\beta_1^*$. We can see that the shape of both functions are similar, but their norms are completely different.  
Hence, the Lasso estimator recovers the true support but gives biased estimators of the coefficients $\beta_j$, $j\in J^*$. 

For the projected estimator $\widehat{\boldsymbol\beta}_{\widehat\lambda,\widehat m}$, as recommended in~\citet{BMR16}, we set the value of the constant $\kappa$ of criterion~\eqref{eq:dim_selec} to $\kappa=2$. The selected dimensions are plotted in Figure~\ref{fig:dim_selec}. We can see that the dimension selected is quite large in general and that it is larger for model 2 than for model 1, which indicates that the dimension selection criterion adapts to the complexity of the model. The resulting estimators are plotted in Figure~\ref{fig:beta1_Lasso_proj}. A similar conclusion as for the projection-free estimator can be drawn concerning the bias problem.

 \begin{figure}
\begin{tabular}{cc}
Example 1 & Example 2\\
\includegraphics[width=0.45\textwidth,height=6cm]{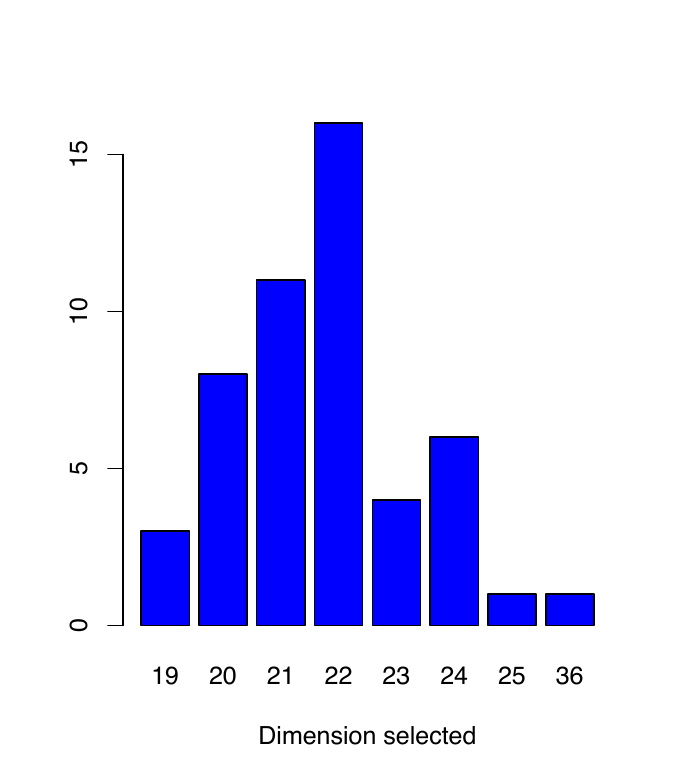}&\includegraphics[width=0.45\textwidth,height=6cm]{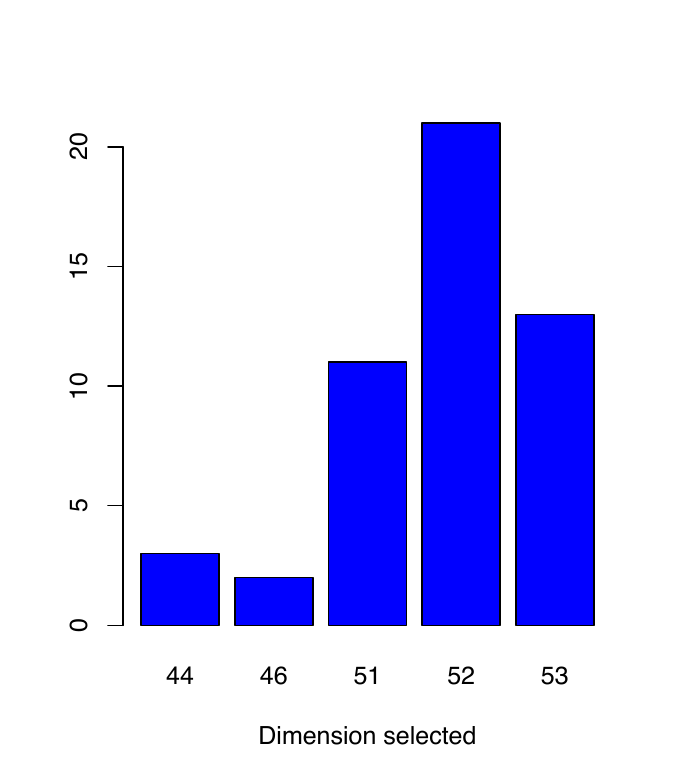}
\end{tabular}
\caption{\label{fig:dim_selec} Bar charts of dimension selected $\widehat m$ over the 50 Monte Carlo replications for the projected estimator $\widehat{\boldsymbol\beta}_{\widehat{\boldsymbol\lambda},\widehat m}$.}
\end{figure}

 \begin{figure}
\begin{tabular}{cc}
Example 1 & Example 2\\
\includegraphics[width=0.45\textwidth,height=6cm]{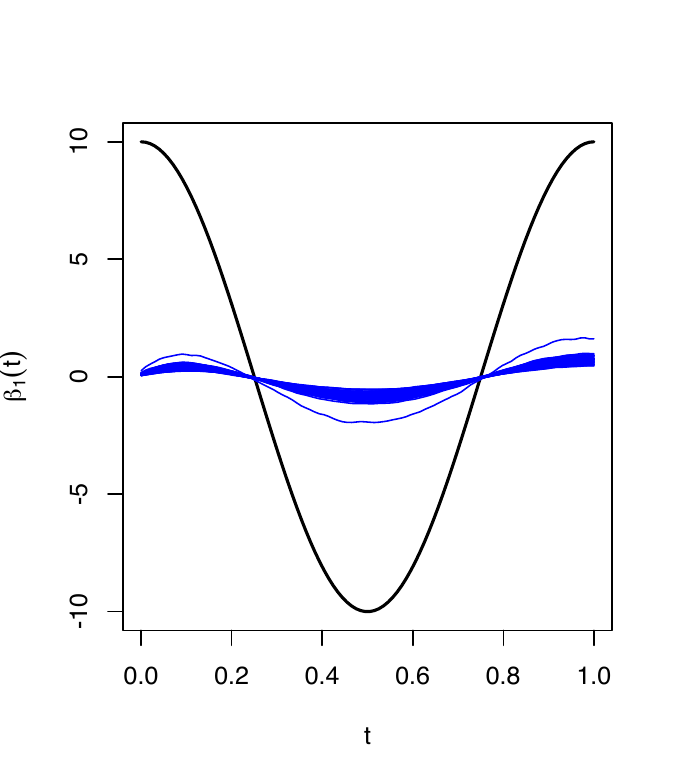}&\includegraphics[width=0.45\textwidth,height=6cm]{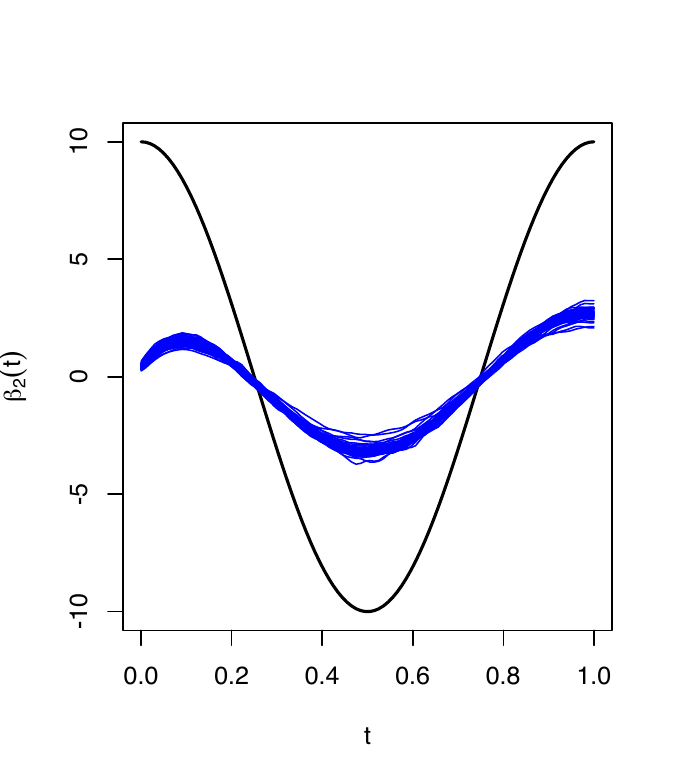}
\end{tabular}
\caption{\label{fig:beta1_Lasso_proj} Plot of $\boldsymbol\beta_1^*$ (solid black line) and 50 Monte-Carlo replications of $\left[\widehat{\boldsymbol\beta}_{\widehat{\boldsymbol\lambda},\widehat m}\right]_1$ (blue lines).}
\end{figure}

\subsection{Final estimator}
\begin{figure}
\begin{tabular}{cc}
Example 1 & Example 2\\
\includegraphics[width=0.45\textwidth]{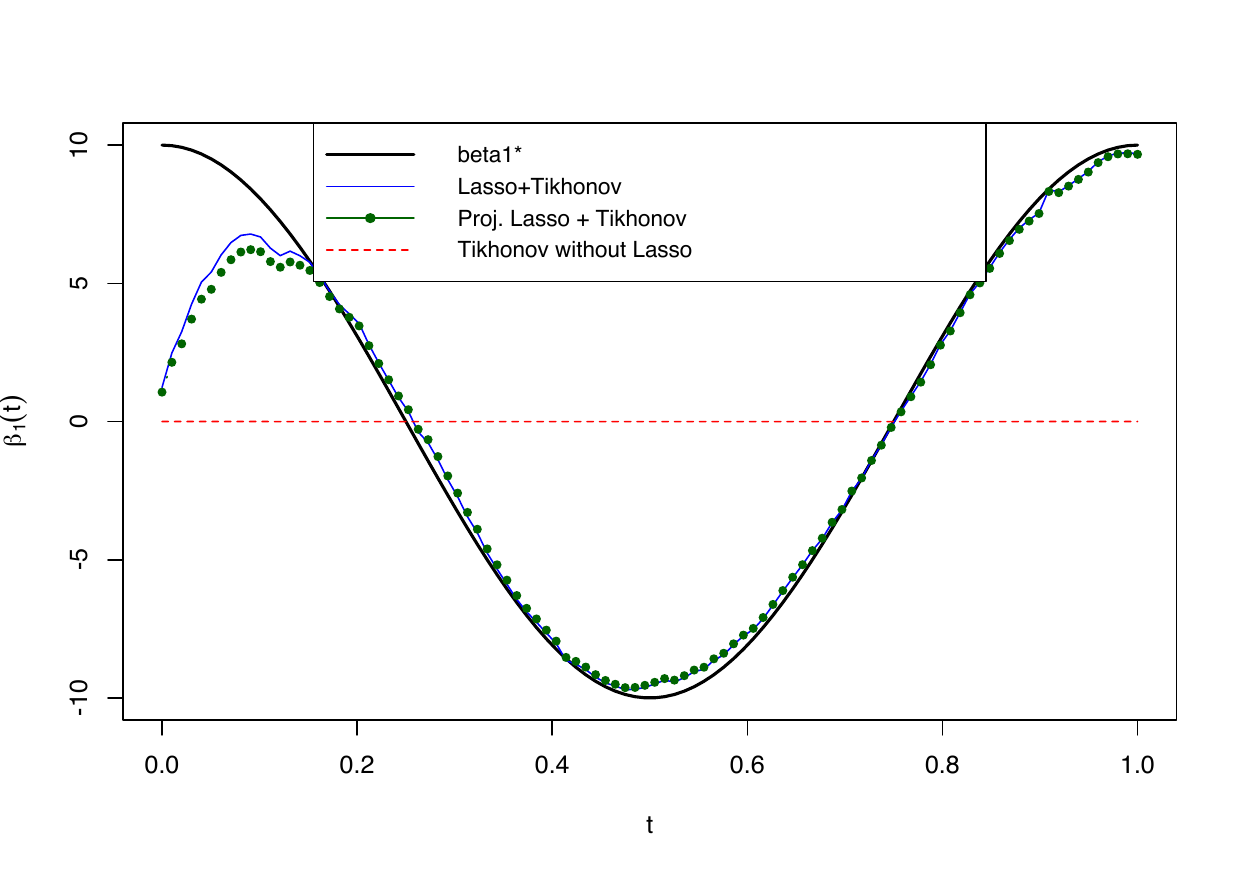}&\includegraphics[width=0.45\textwidth]{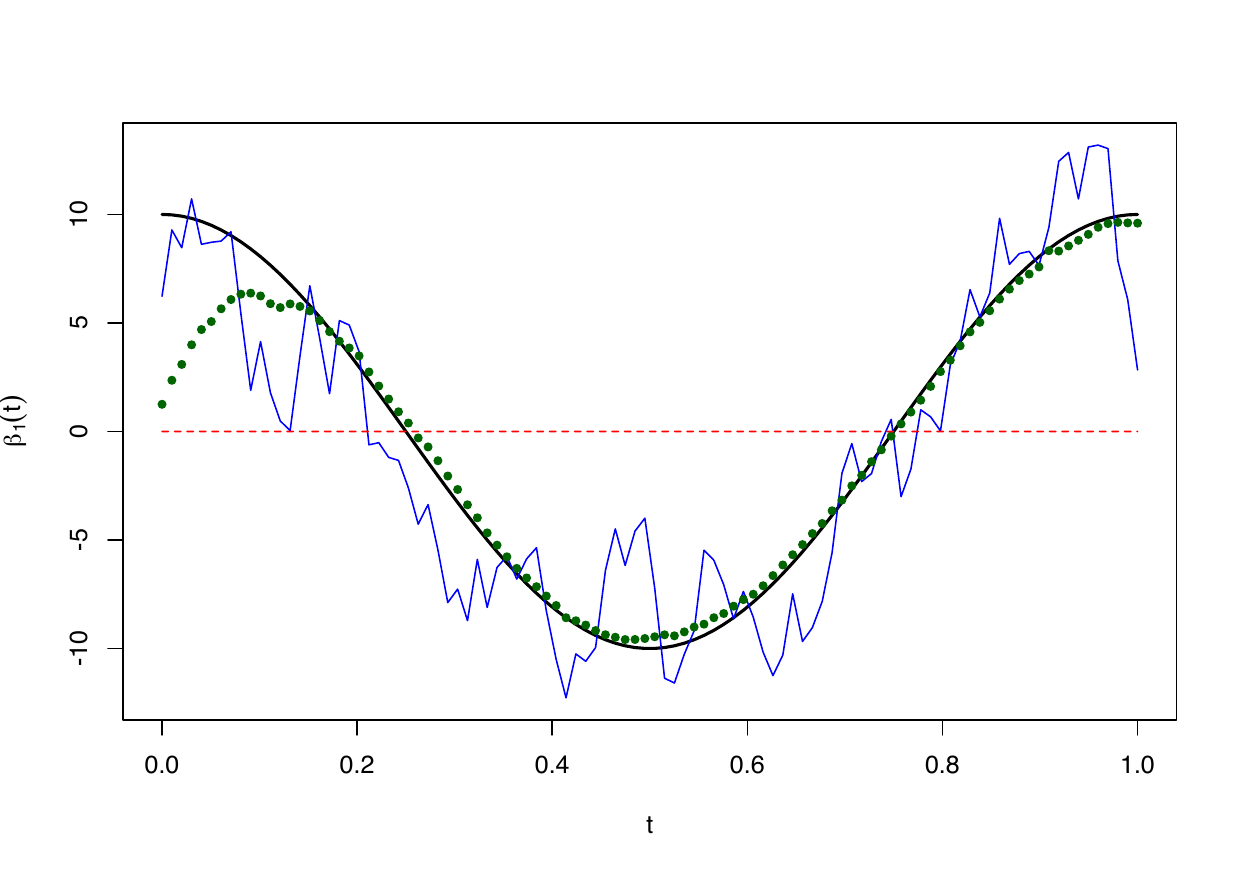}
\end{tabular}
\caption{\label{fig:beta1_Lassotikho} Plot of $\beta_1^*$ (solid black line), the solution of the Tikhonov regularization on the support of the Lasso estimator (dashed blue line) and on the whole support (dotted red line).}
\end{figure}

On Figure~\ref{fig:beta1_Lassotikho} we see that the Tikhonov regularization step reduces the bias in both examples. We can compare it with the effect of Tikhonov regularization step on the whole sample (i.e. without variable selection).  It turns out that, in the case where all the covariates are kept, the algorithm~\eqref{eq:optimtikho} converges very slowly leading to poor estimates. The computation time of the estimators on an iMac 3,06 GHz Intel Core 2 Duo -- with a non optimal code -- are given in Table~\ref{tab:time} for illustrative purposes. 
\begin{table}
\begin{tabular}{|c|c|c|c|}
\hline
		 & Lasso + Tikhonov & Proj. Lasso + Tikhonov & Tikhonov without Lasso\\
\hline
Example 1 & 7.5 min 		& 9.3 min 				&   36.0 min\\
Example 2 & 7.1 min 		& 16.6 min 			&   36.1 min\\
\hline
	\end{tabular}
	\caption{\label{tab:time} Computation time of the estimators. }
\end{table}
 
\subsection{Application to the prediction of energy use of appliances}
\label{sec:applis_elec}

The aim is to study appliances energy consumption -- which is the main source of energy consumption -- in a low energy house situated in Stambruges (Belgium). The data are energy consumption measurements of electrical appliances (\texttt{Appliances}), light (\texttt{light}), temperature and humidity in the kitchen (\texttt{T1} and \texttt{RH1}), in the living room (\texttt{T2} and \texttt{RH2}), in the laundry room (\texttt{T3} and \texttt{RH3}), in the office (\texttt{T4} and \texttt{RH4}), in the bathroom (\texttt{T5} and \texttt{RH5}), outside the building on the north side (\texttt{T6} and \texttt{RH6}), in the ironing room (\texttt{T7} and \texttt{RH7}), in the teenagers' room (\texttt{T8} and \texttt{RH8}) and in the parents' room (\texttt{T9} and \texttt{RH9}) as well as the temperature (\texttt{T$\_$out}), the pressure (\texttt{Press$\_$mm$\_$hg}), the humidity (\texttt{RH$\_$out}), wind speed (\texttt{Windspeed}), visibility (\texttt{Visibility}) and dewpoint temperature (\texttt{Tdewpoint}) from the weather station of Chievres, which is the weather station of the nearest airport.

Each variable is measured every 10 minutes from 11th january, 2016, 5pm to 27th may, 2016, 6pm. 

The data is freely available on UCI Machine Learning Repository (\url{archive.ics.uci.edu/ml/datasets/Appliances+energy+prediction}) and has been studied by~\citet{CFD17}. We refer to this article for a precise description of the experiment and a method to predict appliances energy consumption at a given time from the measurement of the other variables.

Here, we focus on the prediction of the mean appliances energy consumption of one day from the measure of each variable the day before (from midnight to midnight). We then dispose of a dataset of size $n=136$ with $p=24$ functional covariates. Our variable of interest is the logarithm of the mean appliances consumption. In order to obtain better results, we divide the covariates by their range. More precisely, the $j$-th curve of the $i$-th observation $X_i^j$ is transformed as follows 
$$X_i^j(t)\leftarrow\frac{X_i^j(t)}{\max_{i'=1,...,n; t'}X_{i'}^j(t')-\min_{i'=1,...,n; t'}X_{i'}^j(t')}.$$
Recall that usual standardisation techniques are not possible for infinite-dimensional data since the covariance operator of each covariate is non invertible. The choice of the above transformation allows us to obtain covariates of the same order. All the variables are then centered. 

We first plot the evolution of the norm of the coefficients as a function of $r$. The results are shown in Figure~\ref{fig:normes_r_elec}. 
\begin{figure}
\includegraphics[width=0.7\textwidth,height=6cm]{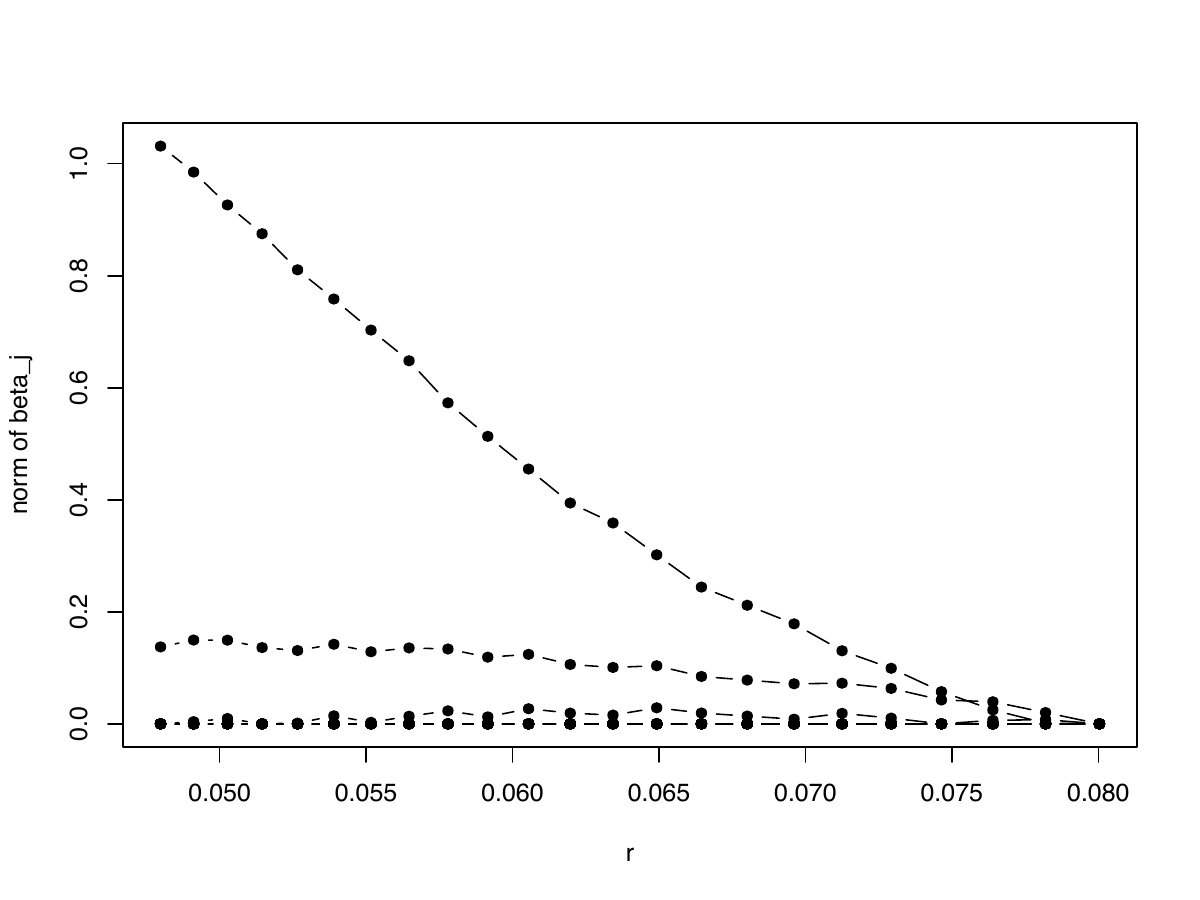}
\caption{\label{fig:normes_r_elec} Plot of the norm of $\left[\widehat{\boldsymbol\beta}_{\widehat{\boldsymbol\lambda},\infty}\right]_j$, for $j=1,...,24$ as a function of $r$.}
\end{figure}

The variables selected by the Lasso criterion are the appliances energy consumption (\texttt{Appliances}), temperature of the laundry room (\texttt{T3}) and temperature of the teenage room (\texttt{T8}) curves. The corresponding slopes are represented in Figure~\ref{fig:beta_elec}. We observe that all the curves take larger values at the end of the day (after 8 pm). This indicates that the values of the three parameters that influence the most the mean appliances energy consumption of the day after are the one measured at the end of the day.  
\begin{figure}
\begin{tabular}{ccc}
\texttt{Appliances}&\texttt{T3}&\texttt{T8}\\
\includegraphics[width=0.3\textwidth]{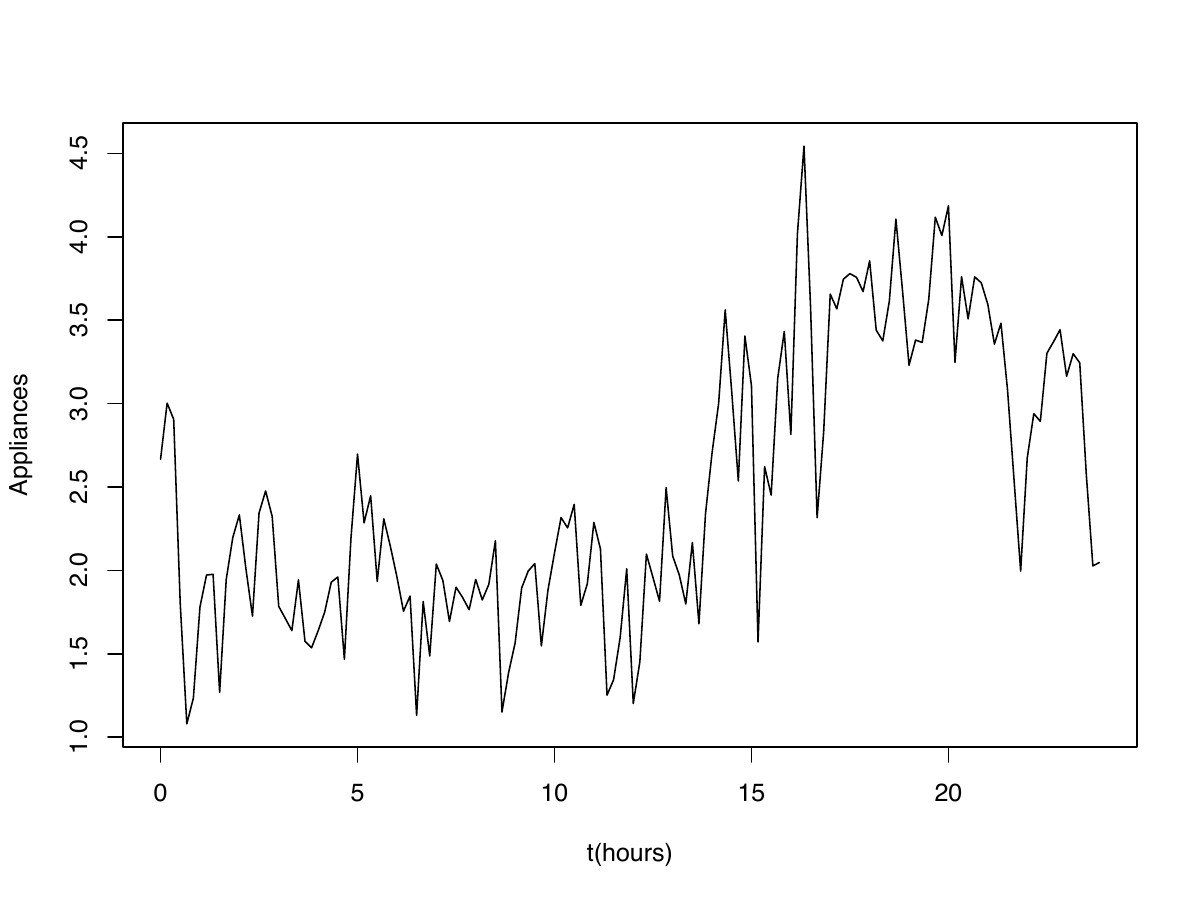}&\includegraphics[width=0.3\textwidth]{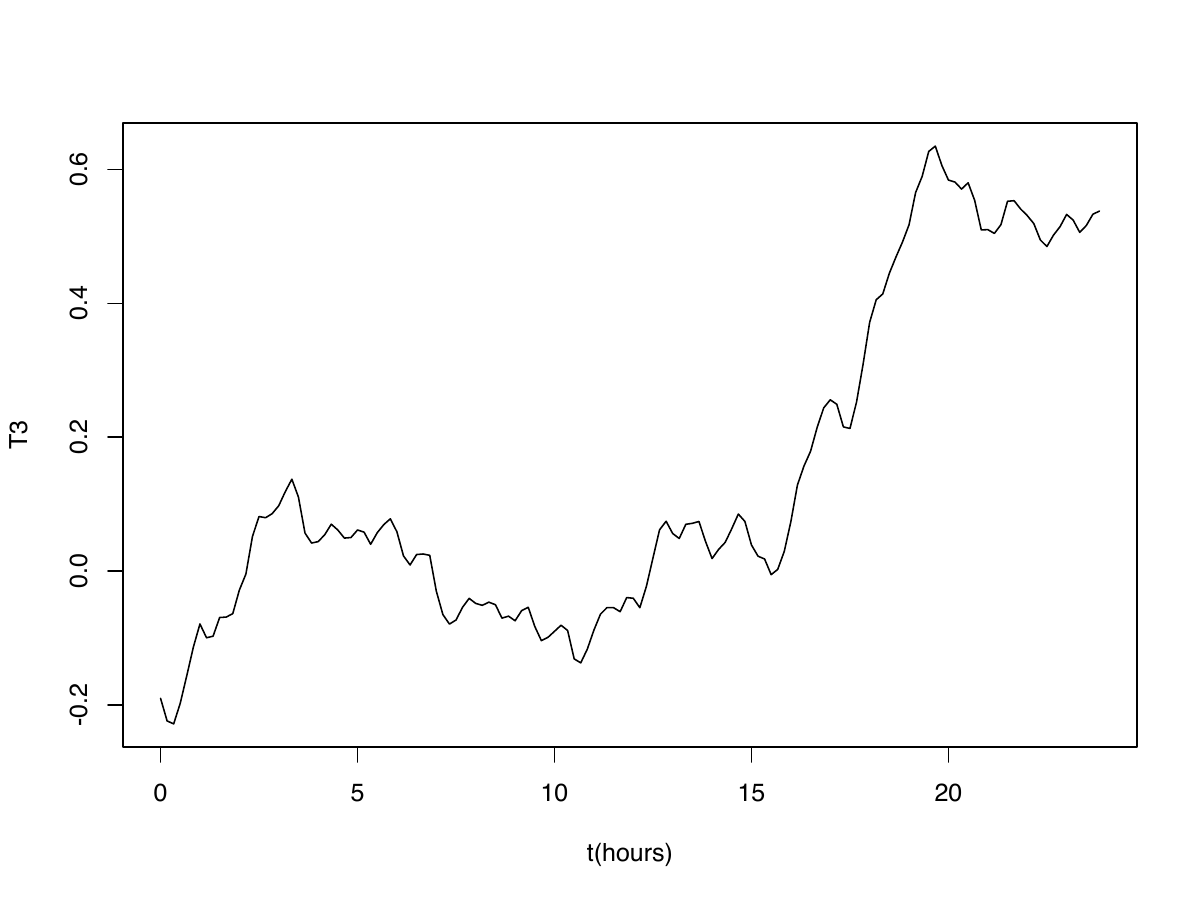}&\includegraphics[width=0.3\textwidth]{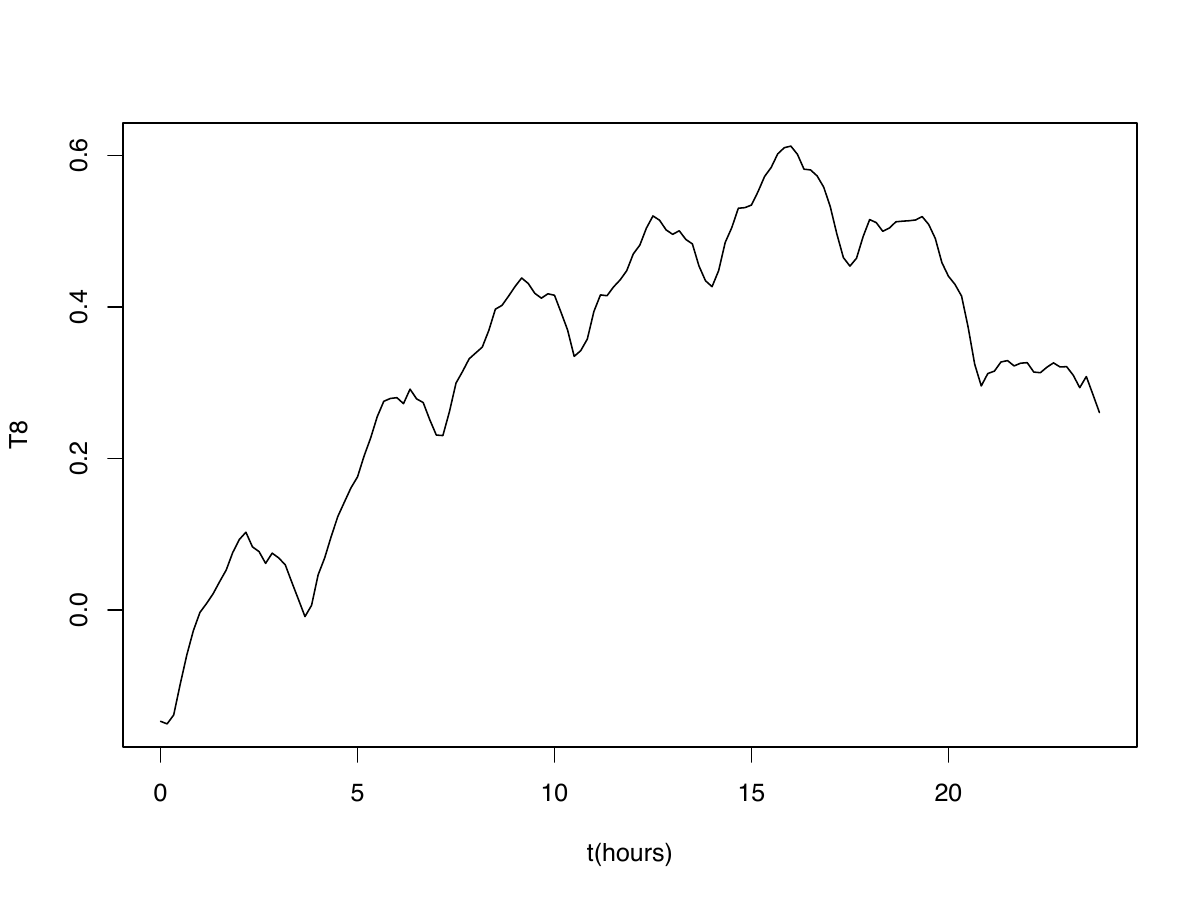}
\end{tabular}
\caption{\label{fig:beta_elec}Plot of the coefficients $\left[\widehat{\boldsymbol\beta}_{\widehat{\boldsymbol\lambda},\infty}\right]_j$ for $j\in J(\widehat{\boldsymbol\beta}_{\widehat{\boldsymbol\lambda},\infty})=\{1,7,17\}$ corresponding to the coefficients associated to the appliance energy consumption curve (\texttt{Appliances}), temperature of the laundry room (\texttt{T3}) and temperature of the teenage room (\texttt{T8}).}
\end{figure} 

%

\section*{Concluding remarks}

The objective of the paper was to study how the theoretical results obtained for Lasso and Group-Lasso penalties can be adapted when the dimension of the covariates is infinite, which is, in particular, the case of functional data. 

\paragraph{Discussion on the theoretical results and open-questions}

As in finite dimension, the main issue is to control the relationship between the empirical norm naturally associated with the least squares criterion, related to the covariance matrix of the covariates, and the norm inducing the sparsity appearing in the penalty. The main problem here is that, as we prove in Lemma~\ref{lem:REfalse}, these two norms cannot be equivalent in infinite dimension. The unprojected estimator seems to have very good performance in practice, and the solution can be computed easily. However, the rates of convergence of this estimator remains an open question. On the other hand, we prove sharp oracle inequalities for the projected estimator and we are able to define a data-based dimension selection criterion that achieves the best trade-off between the bias and the variance term.  However, the rates of convergence of this estimator has not been proven to be optimal. Intuitively, it is not, and it seems likely that an adaptive Lasso procedure is needed to obtain an optimal rate in the minimax sense. These questions, which seem complex questions to solve, are left for future work. 


We could also consider an alternative restricted eigenvalues assumption as it appears in~\citet{JRRSW19} and suppose that there exist two positive numbers $\kappa_1$ and $\kappa_2$ such that
\[
\|\boldsymbol\beta\|_n\geq\kappa_1\|\boldsymbol\beta\|-\kappa_2\|\boldsymbol\beta\|_1, \text{ for all }\boldsymbol\beta\in\mathbf H,
\]
where we denote 
\[
\|\boldsymbol\beta\|_1:=\sum_{j=1}^p\|\beta_j\|_j\text{ for }\boldsymbol\beta=(\beta_1,...,\beta_p)\in\mathbf H. 
\]
This assumption does not suffer from the curse of dimensionality as the assumption $A_{RE(s)}$ does. However, contrary to the finite-dimensional case, the control of the probability that the assumption holds in the random design case is, to our knowledge, still an open question.

\paragraph{Discussion on the numerical results}

From the simulation results, both methods seem to estimate the support of the slope coefficient $\bbeta^*$ well. However, the projected method, which gives us the most accurate theoretical results, is quite difficult to implement in practice, due to the cost of constructing the spaces $\mathbf H^{(m)}$. On the contrary, the unprojected estimator seems to give interesting results, both in the simulation study and in the application on real data. However, the theoretical results (see for instance the remark after Corollary~\ref{cor:rates}) argue for the choice of a finite value of $m$. 

\paragraph{Discussion on the linearity assumption}
The linearity assumption may be too restrictive in some contexts. A natural way to consider a nonlinear regression model is to assume that $Y=m(\mathbf X)+\varepsilon$ where $m:\mathbf H\to\mathbb R$ is an unknown regression function. However, it has been shown by \citet{mas_lower_2012} that, without additional structural assumptions on $m$, this model suffers from the curse of dimensionality which manifests itself here by a very low minimax rate of convergence, typically logarithmic (see also the recent review by \citealt{LV18} and the discussion in \citealt{Geenens11,CR16}). This is also the case for additive models
\[
Y=m_1(X^1)+\hdots+m_p(X^p)+\varepsilon,
\]
with $m_j$ unknown functions $m_j:\mathbb H_j\to\mathbb R$, which could be natural models to consider the sparsity problem. 

This is the reason why semi-parametric models have been introduced and widely studied. In this category, we can mention for example the partially linear models \citep{KXYZ16,WLZ19},
\[
Y=\langle \beta_1,X^1\rangle_1+\hdots+\langle \beta_{p_\infty},X^{p_\infty}\rangle_{p_\infty}+m_1(X^{p_\infty+1})+\hdots+m_p(X^{p})+\varepsilon,
\]
where we recall that $X^{p_\infty+1},\hdots,X^{p}$ are scalar or vector covariates and $X^1,\hdots,X^{p_\infty}$ are functional covariates. The approach developed in this paper could be directly extended to this model by considering estimators by projection of $m_j$, as in \citet{BTW07}. However, this introduces an additional bias that needs to be handled in the theoretical results and requires careful selection of the projection spaces and their dimensions. 

This model has been generalized, for example, to the case of single-index models (see \citealt{NAV21} and references cited). 
\[
Y=g_1(\langle \beta_1,X^1\rangle_1)+\hdots+g_{p_\infty}(\langle \beta_{p_\infty},X^{p_\infty}\rangle_{p_\infty})+m_1(X^{p_\infty+1})+\hdots+m_p(X^{p})+\varepsilon,
\]
where the $g_j$'s are unknown real functions. This type of model, poses theoretical questions more difficult to solve than the previous one, because the coefficients $\beta_j$ do not depend linearly on the observations.   


\paragraph{Acknowledgements}

I would like to thank Vincent Rivoirard and Ga\"elle Chagny for their helpful advices and careful reading of the manuscript. 
The research is partly supported by the french Agence Nationale de la Recherche (ANR-18-CE40-0014 projet SMILES). 

 \appendix
\section{Proofs}
\subsection{ Proof of Lemma~\ref{lem:REfalse}}
\label{subsec:proofREfalse}

\begin{proof}
Let $J\subset\{1,\hdots,p\}$ such that  $\dim(\mathbf H_J)>{\rm rk}(\widehat{\boldsymbol\Gamma}_J)$. This implies that $\dim(\rm{\ker}(\widehat{\boldsymbol\Gamma}_J))\geq 1$ and then that there exists $\boldsymbol\delta_J=(\delta_j)_{j\in J}\in\mathbf H_J\backslash\{0\}$ such that $\widehat{\boldsymbol\Gamma}_J\boldsymbol\delta_J=0$. Define now from  $\boldsymbol\delta_J$, $\boldsymbol\delta=(\delta_1,\hdots,\delta_p)\in\mathbf H$ such that $\delta_j=0$ if $j\notin J$.
 
Recall the definition of the operator 
\begin{eqnarray*}
\widehat{\boldsymbol\Gamma}_J:&\mathbf H_J&\to\mathbf H_J\\
						   & \boldsymbol\beta=(\beta_j)_{j\in J}&\mapsto \left(\frac1n\sum_{i=1}^n\sum_{j\in J}\langle \beta_j,X_i^j\rangle_j X_i^{j'}\right)_{j'\in J},
\end{eqnarray*}
and observe that
\[
\|\boldsymbol\delta\|_n^2=\langle \widehat{\boldsymbol\Gamma}\boldsymbol\delta,\boldsymbol\delta\rangle=0. 
\]
Moreover, $\boldsymbol\delta$ satisfies the constraints 
\[
0=\sum_{j\notin J}\lambda_j\|\delta_j\|_j\leq c_0\sum_{j\in J}\lambda_j\|\delta_j\|_j,
\]
for all choices of $\lambda_1,\hdots,\lambda_p$ and for all $c_0>0$ which ends the proof.
 \end{proof}
 
 \subsection{Proof of Proposition~\ref{prop:oracle}}
 \label{subsec:proof:prop:oracle}
\begin{proof}

We prove only~\eqref{eq:oracle}, Inequality~\eqref{eq:oracle_proj} follows the same lines. The proof below is largely inspired by the proof of~\citet{lounici_oracle_2011}, using the improvement of \citet{BT17} to obtain a sharp oracle inequality. First remark that some algebra gives us the following result, true for all $\boldsymbol\beta\in\mathbf H^{(m)}$,
\begin{equation}\label{eq:oracle:start}
\n{\estim-\bbeta^*}_n^2-\n{\bbeta-\bbeta^*}_n^2=\frac2n\sum_{i=1}^n\ps{\estim-\bbeta^*}{\mathbf X_i}\ps{\estim-\bbeta}{\mathbf X_i}-\n{\estim-\bbeta}_n^2
\end{equation}

We can easily verify that the function 
$$\gamma:\boldsymbol\beta\in\mathbf H\mapsto  \frac1n\sum_{i=1}^n\left(Y_i-\ps{\boldsymbol\beta}{\mathbf X_i}\right)^2+2\sum_{j=1}^p\lambda_j\n{\beta_j}_j=\gamma_1(\boldsymbol\beta)+\gamma_2(\boldsymbol\beta)$$
is a proper convex function. Hence, $\estim$ is a minimum of $\gamma$ over $\mathbf H$ if and only if 0 is a subgradient $\partial\gamma(\estim)$ of $\gamma$ at the point $\estim$. 

The function $\gamma_1:\boldsymbol\beta\mapsto  \frac1n\sum_{i=1}^n\left(Y_i-\ps{\boldsymbol\beta}{\mathbf X_i}\right)^2$ is differentiable on $\mathbf H$, with gradient,
$$\left(-\frac2n\sum_{i=1}^n(Y_i-\ps{\boldsymbol\beta}{X_i})X_i^j\right)_{1\leq j\leq p}=-\frac2n\sum_{i=1}^n(Y_i-\ps{\boldsymbol\beta}{X_i})\mathbf X_i$$
and $\gamma_2:\boldsymbol\beta\mapsto2\sum_{j=1}^p\lambda_j\n{\beta_j}_j$ is differentiable on $D:=\{\boldsymbol\beta=(\boldsymbol\beta_1,...,\boldsymbol\beta_p)\in\mathbf H,\ \forall j=1,...,p,\ \beta_j\neq0\}$ with gradient 
$$\left(2\lambda_j\frac{\beta_j}{\n{\beta_j}_j}\right)_{1\leq j\leq p}.$$
Since, for all $j=1,...,p$, the subdifferential of $\n{\cdot}_j$ at the point 0 is the closed unit ball of $\mathbf H_j$, the subdifferential of  $\gamma_2:\boldsymbol\beta\mapsto2\sum_{j=1}^p\lambda_j\n{\beta_j}_j$ at the point $\boldsymbol\beta\in D^c$, is the set
\begin{equation}\label{eq:subdiff}
\partial\gamma_2(\boldsymbol\beta)=\left\{\boldsymbol\delta=(\delta_1,...,\delta_p)\in\mathbf H, \delta_j=2\lambda_j\frac{\beta_j}{\n{\beta_j}_j} \text{ if }\beta_j\neq 0, \n{\delta_j}_j\leq 2\lambda_j\text{ if }\beta_j= 0\right\}.
\end{equation}

Hence, the subdifferential of $\gamma$ at the point $\boldsymbol\beta=(\boldsymbol\beta_1,...,\boldsymbol\beta_p)\in\mathbf H$ is the set
$$\partial\gamma(\boldsymbol\beta)=\left\{\boldsymbol\theta\in\mathbf H,\ \exists\boldsymbol\delta\in\partial\gamma_2(\boldsymbol\beta),\ \boldsymbol\theta=-\frac2n\sum_{i=1}^n(Y_i-\ps{\boldsymbol\beta}{\mathbf X_i})\mathbf X_i+\boldsymbol\delta\right\}.$$
Then, since $0\in\partial\gamma(\estim)$,  we know that there exists $\widehat{\boldsymbol\delta}=(\widehat\delta_1,\hdots,\widehat\delta_p)\in\partial\gamma_2(\estim)$ such that 
\[
 0=-\frac2n\sum_{i=1}^n(Y_i-\ps{\estim}{\mathbf X_i})\mathbf X_i+\widehat{\boldsymbol\delta}=-\frac2n\sum_{i=1}^n(\ps{\bbeta^*}{\mathbf X_i}+\varepsilon_i-\ps{\estim}{\mathbf X_i})\mathbf X_i+\widehat{\boldsymbol\delta}. 
\]
Then 
\[
\frac2n\sum_{i=1}^n\ps{\estim-\bbeta^*}{\mathbf X_i}\mathbf X_i=\frac2n\sum_{i=1}^n\varepsilon_i\mathbf X_i-\widehat{\boldsymbol\delta}
\]
which implies
\begin{equation}\label{eq:oracle:start1}
\frac2n\sum_{i=1}^n\ps{\estim-\bbeta^*}{\mathbf X_i}\ps{\estim-\bbeta}{\mathbf X_i}=\frac2n\sum_{i=1}^n\varepsilon_i\ps{\estim-\bbeta}{\mathbf X_i}+\ps{\bbeta-\estim}{\widehat{\boldsymbol\delta}}.
\end{equation}
Now remark that, denoting $[\widehat\beta_{\boldsymbol\lambda,\infty}]_j$ the $j$-th coordinate of $\estim$ we have, by definition of $\widehat{\boldsymbol\delta}$, 
\begin{equation}\label{eq:oracle:start2}
\begin{split}
\ps{\bbeta-\estim}{\widehat{\boldsymbol\delta}}&=\sum_{j=1}^p\ps{\beta_j-[\widehat\beta_{\boldsymbol\lambda,\infty}]_j}{\widehat{\delta}_j}_j=\sum_{j=1}^p\ps{\beta_j}{\widehat{\delta}_j}_j-2\lambda_j\n{[\widehat\beta_{\boldsymbol\lambda,\infty}]_j}_j\\
&\leq 2\sum_{j=1}^p\lambda_j\left(\n{\beta_j}_j-\n{[\widehat\beta_{\boldsymbol\lambda,\infty}]_j}_j\right).
\end{split}
\end{equation}
Then inserting~\eqref{eq:oracle:start1} and \eqref{eq:oracle:start2} in~\eqref{eq:oracle:start}, we get
\begin{eqnarray}
\n{\estim-\bbeta^*}_n^2-\n{\bbeta-\bbeta^*}_n^2&\leq& \frac2n\sum_{i=1}^n\varepsilon_i\ps{\bbeta-\estim}{\mathbf X_i}+2\sum_{j=1}^p\lambda_j\left(\n{\beta_j}_j-\n{[\widehat\beta_{\boldsymbol\lambda,\infty}]_j}_j\right)\nonumber\\
&&-\n{\estim-\bbeta}_n^2. \label{eq:oracle:start:fin}
\end{eqnarray}
Then the key result proven by~\citet[Lemma A.2]{BLT18} in a finite-dimensional context also holds in our infinite-dimensional context. 


We now deal with the term involving the $\varepsilon_i$'s. Remark that, writing $\widehat{\boldsymbol\beta}_{\boldsymbol\lambda,\infty}=\widehat{\boldsymbol\beta}_{\boldsymbol\lambda,\infty}^{(m)}+\widehat{\boldsymbol\beta}_{\boldsymbol\lambda,\infty}^{(\perp m)}$ where $\widehat{\boldsymbol\beta}^{(m)}_{\boldsymbol\lambda,\infty}$ denotes the orthogonal projection of $\widehat{\boldsymbol\beta}_{\boldsymbol\lambda,\infty}$ onto $\mathbf H^{(m)}$ and $\widehat{\boldsymbol\beta}^{(\perp m)}_{\boldsymbol\lambda,\infty}$ denotes the orthogonal projection of $\widehat{\boldsymbol\beta}_{\boldsymbol\lambda,\infty}$ onto $\left(\mathbf H^{(m)}\right)^{\perp}$,
\begin{eqnarray*}
\frac1n\sum_{i=1}^n\varepsilon_i\ps{\bbeta-\estim}{\mathbf X_i}&=&\ps{\bbeta-\estim}{\frac 1n\sum_{i=1}^n\varepsilon_i\mathbf X_i} \\
&=&\ps{\boldsymbol\beta-\estim^{(m)}}{\frac 1n\sum_{i=1}^n\varepsilon_i\mathbf X_i}+\ps{\widehat{\boldsymbol\beta}_{\boldsymbol\lambda,\infty}^{(\perp m)}}{-\frac 1n\sum_{i=1}^n\varepsilon_i\mathbf X_i}\\
&\leq& \sum_{j=1}^p\ps{\beta_j-[\estim^{(m)}]_j}{\frac 1n\sum_{i=1}^n\varepsilon_iX_i^j}_j+\n{\widehat{\boldsymbol\beta}_{\boldsymbol\lambda,\infty}^{(\perp m)}}\n{\frac 1n\sum_{i=1}^n\varepsilon_i\mathbf X_i}\\
&\leq& \sum_{j=1}^p\n{[\widehat\beta_{\boldsymbol\lambda,\infty}^{(m)}]_j-\beta_j}_j\n{\frac1n\sum_{i=1}^n\varepsilon_iX_i^j}_j+\n{\widehat{\boldsymbol\beta}_{\boldsymbol\lambda,\infty}^{(\perp m)}}\n{\frac 1n\sum_{i=1}^n\varepsilon_i\mathbf X_i}.
\end{eqnarray*}

Let $\mathcal A = \bigcap_{j=1}^p\mathcal A_j$, with 
\begin{equation*}
\mathcal A_j = \left\{\n{\frac1n\sum_{i=1}^n\varepsilon_iX_i^j}_j\leq\lambda_j/2\right\}.
\end{equation*}
On the set $\mathcal A$,
\begin{equation}\label{eq:oracle1b}
\frac2n\sum_{i=1}^n\varepsilon_i\ps{\bbeta-\estim}{\mathbf X_i}\leq \sum_{j=1}^p\lambda_j\n{[\widehat\beta_{\boldsymbol\lambda,\infty}^{(m)}]_j-\beta_j}_j+\n{\widehat{\boldsymbol\beta}_{\boldsymbol\lambda,\infty}^{(\perp m)}}\sqrt{\sum_{j=1}^p\lambda_j^2}.
\end{equation}
%
Since the projector $\Pi_m$ verifies $C_{supp}$, 
$$
\n{[\widehat\beta_{\boldsymbol\lambda,\infty}^{(m)}]_j}_j=\n{\pi_j\Pi_m\widehat\beta_{\boldsymbol\lambda,\infty}^{(m)}}=\n{\Pi_m\pi_j\widehat\beta_{\boldsymbol\lambda,\infty}^{(m)}}\leq\n{\pi_j\widehat\beta_{\boldsymbol\lambda,\infty}}=\n{[\widehat\beta_{\boldsymbol\lambda,\infty}]_j}_j
$$
and gathering equations~\eqref{eq:oracle:start:fin} and~\eqref{eq:oracle1b},
%
\begin{equation*}
\begin{split}
\n{\estim-\bbeta^*}_n^2&-\n{\bbeta-\bbeta^*}_n^2+\sum_{j=1}^p\lambda_j\n{[\widehat\beta_{\boldsymbol\lambda,\infty}^{(m)}]_j-\beta_j}_j\\
&\hspace{-2cm}\leq 2\sum_{j=1}^p\lambda_j\left(\n{[\widehat\beta_{\boldsymbol\lambda,\infty}^{(m)}]_j-\beta_j}_j+\n{\beta_j}_j-\n{[\widehat\beta_{\boldsymbol\lambda,\infty}]_j}_j\right)-\n{\estim-\bbeta}_n^2+\n{\widehat{\boldsymbol\beta}_{\boldsymbol\lambda,\infty}^{(\perp m)}}\sqrt{\sum_{j=1}^p\lambda_j^2}
\\
&\hspace{-2cm}\leq 2\sum_{j=1}^p\lambda_j\left(\n{[\widehat\beta_{\boldsymbol\lambda,\infty}^{(m)}]_j-\beta_j}_j+\n{\beta_j}_j-\n{[\widehat\beta_{\boldsymbol\lambda,\infty}^{(m)}]_j}_j\right)-\n{\estim-\bbeta}_n^2+\n{\widehat{\boldsymbol\beta}_{\boldsymbol\lambda,\infty}^{(\perp m)}}\sqrt{\sum_{j=1}^p\lambda_j^2}
\\
&\hspace{-2cm}\leq 4\!\sum_{j\in J(\boldsymbol\beta)}\lambda_j\n{[\widehat\beta_{\boldsymbol\lambda,\infty}^{(m)}]_j-\beta_j}_j
-\n{\estim-\bbeta}_n^2+\n{\widehat{\boldsymbol\beta}_{\boldsymbol\lambda,\infty}^{(\perp m)}}\sqrt{\sum_{j=1}^p\lambda_j^2},
\end{split}
\end{equation*}
since  $\n{\beta_j}_j-\n{[\widehat\beta_{\boldsymbol\lambda,\infty}^{(m)}]_j}_j\leq\n{\beta_j-[\widehat\beta_{\boldsymbol\lambda,\infty}^{(m)}]_j}_j$. 
Finally 
\begin{equation}\label{eq:debut_oracle}
\begin{split}
\n{\estim-\bbeta^*}_n^2-\n{\bbeta-\bbeta^*}_n^2\leq&\ 3\!\sum_{j\in J(\boldsymbol\beta)}\lambda_j\n{[\widehat\beta_{\boldsymbol\lambda,\infty}^{(m)}]_j-\beta_j}_j-\!\sum_{j\notin J(\boldsymbol\beta)}\lambda_j\n{[\widehat\beta_{\boldsymbol\lambda,\infty}^{(m)}]_j-\beta_j}_j\\
&-\n{\estim-\bbeta}_n^2+\n{\widehat{\boldsymbol\beta}_{\boldsymbol\lambda,\infty}^{(\perp m)}}\sqrt{\sum_{j=1}^p\lambda_j^2}
\end{split}
\end{equation}
%
We consider now two cases : 
\begin{enumerate}
\item $3\sum_{j\in J(\beta)}\lambda_j\n{[\widehat\beta_{\boldsymbol\lambda,\infty}^{(m)}]_j-\beta_j}_j\geq \sum_{j\notin J(\beta)}\lambda_j\n{[\widehat\beta_{\boldsymbol\lambda,\infty}^{(m)}]_j-\beta_j}_j$. 
\item $3\sum_{j\in J(\beta)}\lambda_j\n{[\widehat\beta_{\boldsymbol\lambda,\infty}^{(m)}]_j-\beta_j}_j < \sum_{j\notin J(\beta)}\lambda_j\n{[\widehat\beta_{\boldsymbol\lambda,\infty}^{(m)}]_j-\beta_j}_j$. 
\end{enumerate}

First remark that in case 2., the result is obvious. Now, in case 1., we have, 
by definition of $\tilde\kappa_n^{(m)}(s)$, 
\[
 \tilde\kappa_n^{(m)}(s)\leq \frac{\n{\boldsymbol{\widehat\beta}_{\boldsymbol\lambda,\infty}^{(m)}-\boldsymbol\beta}_n}{\sqrt{\sum_{j\in J(\boldsymbol\beta)}\n{[\widehat\beta_{\boldsymbol\lambda,\infty}^{(m)}]_j-\beta_j}_j^2}}
\]
or equivalently
$$\sqrt{\sum_{j\in J(\boldsymbol\beta)}\n{[\widehat\beta_{\boldsymbol\lambda,\infty}^{(m)}]_j-\beta_j}_j^2}\leq \frac{1}{\widetilde\kappa_n^{(m)}(s)}\n{\widehat{\boldsymbol\beta}_{\boldsymbol\lambda,\infty}^{(m)}-\boldsymbol\beta}_n.$$

Then, using twice the fact that, for all $x,y\in\mathbb R$, $3xy\leq  x^2+(9/4)y^2$, Equation~\eqref{eq:debut_oracle} becomes,

\begin{equation*}
\begin{split}
\n{\estim-\bbeta^*}_n^2&-\n{\bbeta-\bbeta^*}_n^2 \\
&\leq 3\sqrt{\sum_{j\in J(\boldsymbol\beta)}\lambda_j^2}\sqrt{\sum_{j\in J(\boldsymbol\beta)}\n{[\widehat\beta_{\boldsymbol\lambda,\infty}^{(m)}]_j-\beta_j}_j^2}-\n{\estim-\bbeta}_n^2+\n{\widehat{\boldsymbol\beta}_{\boldsymbol\lambda,\infty}^{(\perp m)}}\sqrt{\sum_{j=1}^p\lambda_j^2},\\
&\leq\frac{3}{\widetilde\kappa_n^{(m)}(s)}\sqrt{\sum_{j\in J(\boldsymbol\beta)}\lambda_j^2}\n{\widehat{\boldsymbol\beta}_{\boldsymbol\lambda,\infty}^{(m)}-\boldsymbol\beta}_n-\n{\estim-\bbeta}_n^2+\n{\widehat{\boldsymbol\beta}_{\boldsymbol\lambda,\infty}^{(\perp m)}}\sqrt{\sum_{j=1}^p\lambda_j^2},\\
&\leq\frac{3}{\widetilde\kappa_n^{(m)}(s)}\sqrt{\sum_{j\in J(\boldsymbol\beta)}\lambda_j^2}\left(\n{\widehat{\boldsymbol\beta}_{\boldsymbol\lambda,\infty}-\boldsymbol\beta^*}_n+\n{\widehat{\boldsymbol\beta}_{\boldsymbol\lambda,\infty}^{(\perp m)}}_n\right)\\
&\qquad-\n{\estim-\bbeta}_n^2+\n{\widehat{\boldsymbol\beta}_{\boldsymbol\lambda,\infty}^{(\perp m)}}\sqrt{\sum_{j=1}^p\lambda_j^2}\\
&\leq \n{\widehat{\boldsymbol\beta}_{\boldsymbol\lambda,\infty}-\boldsymbol\beta^*}_n^2+\frac{9}{4(\widetilde\kappa_n^{(m)}(s))^2}\sum_{j\in J(\boldsymbol\beta)}\lambda_j^2-\n{\estim-\bbeta}_n^2+R_{n,m}\\\end{split},
\end{equation*}
where we recall that 
\[
R_{n,m}=\left(\n{\widehat{\boldsymbol\beta}_{\boldsymbol\lambda,\infty}^{(\perp m)}}+\frac{3}{\widetilde\kappa_n^{(m)}(s)}\n{\widehat{\boldsymbol\beta}_{\boldsymbol\lambda,\infty}^{(\perp m)}}_n\right)\sqrt{\sum_{j=1}^p\lambda_j^2},
\]
which implies the expected result.

We turn now to the upper-bound on the probability of the complement of the event $\mathcal A$. 
Conditionally to $\mathbf X_1,\hdots,\mathbf X_n$, since $\{\varepsilon_i\}_{1\leq i\leq n}\sim_{i.i.d}\mathcal N(0,\sigma^2)$, the variable $\frac1n\sum_{i=1}^n\varepsilon_iX_i^j$ is a Gaussian random variable taking values in the Hilbert (hence Banach) space $\mathbb H_j$. Therefore, from Proposition~\ref{prop:LT91}, we know that, denoting $\mathbb P_{\mathbf X}(\cdot)=\mathbb P(\cdot|\mathbf X_1,\hdots,\mathbf X_n)$ and $\mathbb E_{\mathbf X}[\cdot]=\mathbb E[\cdot|\mathbf X_1,\hdots,\mathbf X_n]$, 
\begin{equation*}
\mathbb P_{\mathbf X}(\mathcal A_j^c)\leq 4\exp\left(-\frac{\lambda_j^2}{32\mathbb E_{\mathbf X}\left[\n{\frac1n\sum_{i=1}^n\varepsilon_iX_i^j}_j^2\right]}\right)=\exp\left(-\frac{nr_n^2}{32\sigma^2}\right), 
\end{equation*}
since $\lambda_j^2=r_n^2 \frac1n\sum_{i=1}^n\n{X_i^j}_j^2$ and
$$\mathbb E_{\mathbf X}\left[\n{\frac1n\sum_{i=1}^n\varepsilon_iX_i^j}_j^2\right]=\frac1{n^2}\sum_{i_1,i_2=1}^n\mathbb E_{\mathbf X}\left[\varepsilon_{i_1}\varepsilon_{i_2}\ps{X_{i_1}^j}{X_{i_2}^2}_j\right]=\frac{\sigma^2}{n}\frac1n\sum_{i=1}^n\n{X_i^j}_j^2.$$


This implies that 
$$\mathbb P(\mathcal A^c)\leq p\exp\left(-\frac{nr_n^2}{32\sigma^2}\right)\leq p^{1-q},$$
as soon as $r_n\geq 4\sqrt2\sigma\sqrt{q\ln(p)/n}$.

\end{proof}
\subsection{Proof of Theorem~\ref{thm:oracle}}
\label{subsec:proof:thm:oracle}
\begin{proof}
By definition of $\widehat m$, we know that, for all $m=1,\hdots,N_n$,
\[
\frac1n\sum_{i=1}^n\left(Y_i-\langle\widehat{\boldsymbol\beta}_{\boldsymbol\lambda,\widehat m},\mathbf X_i\rangle\right)^2+\kappa\sigma^2\frac{\widehat m}n\log(n)\leq \frac1n\sum_{i=1}^n\left(Y_i-\langle\widehat{\boldsymbol\beta}_{\boldsymbol\lambda,m},\mathbf X_i\rangle\right)^2+\kappa\sigma^2\frac{m}n\log(n). 
\]
Now, we decompose the quantity, for all $\bbeta\in\mathbf H$,
\begin{eqnarray*}
\frac1n\sum_{i=1}^n\left(Y_i-\langle\boldsymbol\beta,\mathbf X_i\rangle\right)^2&=&\frac1n\sum_{i=1}^n\left(\langle\boldsymbol\beta^*-\boldsymbol\beta,\mathbf X_i\rangle+\varepsilon_i\right)^2\nonumber\\
&=&\left\|\boldsymbol\beta^*-\boldsymbol\beta\right\|_n^2+\frac2n\sum_{i=1}^n\varepsilon_i\langle\boldsymbol\beta^*-\boldsymbol\beta,\mathbf X_i\rangle+\frac1n\sum_{i=1}^n\varepsilon_i^2.
\end{eqnarray*}
and we obtain 
\begin{eqnarray}
\left\|\widehat{\boldsymbol\beta}_{\lambda,\widehat m}-\boldsymbol\beta^*\right\|_n^2&\leq&\left\|\widehat{\boldsymbol\beta}_{\lambda,m}-\bbeta^*\right\|_n^2+\frac2n\sum_{i=1}^n\varepsilon_i\langle\widehat{\boldsymbol\beta}_{\boldsymbol\lambda,\widehat m}-\widehat{\boldsymbol\beta}_{\lambda,m},\mathbf X_i\rangle\nonumber\\
&&+\kappa\sigma^2\frac{m}{n}\log(n)-\kappa\sigma^2\frac{\widehat m}{n}\log(n)\label{eq:oracle1}. 
\end{eqnarray}

%

Let $1/2>\eta>0$, since $2xy\leq \eta x^2+\eta^{-1} y^2$ for all $x,y\in\mathbb R$, 
\[
\frac2n\sum_{i=1}^n\varepsilon_i\langle\widehat{\boldsymbol\beta}_{\boldsymbol\lambda,\widehat m}-\widehat{\boldsymbol\beta}_{\boldsymbol\lambda,m},\mathbf X_i\rangle\leq \eta\left\|\widehat{\boldsymbol\beta}_{\boldsymbol\lambda,\widehat m}-\widehat{\boldsymbol\beta}_{\boldsymbol\lambda,m}\right\|_n^2+\eta^{-1}\nu_n^2\left(\frac{\widehat{\boldsymbol\beta}_{\boldsymbol\lambda,\widehat m}-\widehat{\boldsymbol\beta}_{\boldsymbol\lambda,m}}{\left\|\widehat{\boldsymbol\beta}_{\boldsymbol\lambda,\widehat m}-\widehat{\boldsymbol\beta}_{\boldsymbol\lambda,m}\right\|_n}\right),
\]
where $\nu_n^2(\cdot):=\frac1n\sum_{i=1}^n\varepsilon_i\langle \cdot,X_i\rangle$.
Now, we define the set :
\begin{equation}\label{eq:defBm}
\mathcal B_m:=\bigcap_{m'=1}^{N_n}\left\{\sup_{f\in\mathbf H^{(\max\{m,m'\})}, \|f\|_n=1}\nu_n^2(f)< \frac\kappa{2\eta^{-1}}\log(n)\sigma^2\frac{\max\{m,m'\}}{n}\right\}.
\end{equation}
On the set $\mathcal B_m$, since $\widehat{\boldsymbol\beta}_{\boldsymbol\lambda,\widehat m}-\widehat{\boldsymbol\beta}_{\boldsymbol\lambda,m}\in\mathbf H^{(\max\{m,\widehat m\})}$,
\begin{eqnarray}
\frac2n\sum_{i=1}^n\varepsilon_i\langle\widehat{\boldsymbol\beta}_{\boldsymbol\lambda,\widehat m}-\widehat{\boldsymbol\beta}_{\boldsymbol\lambda,m},\mathbf X_i\rangle&\leq& \eta\left\|\widehat{\boldsymbol\beta}_{\boldsymbol\lambda,\widehat m}-\widehat{\boldsymbol\beta}_{\boldsymbol\lambda,m}\right\|_n^2+\eta^{-1}\sup_{f\in\mathbf H^{(\max\{m,\widehat m\})},\|f\|_n^2=1}\nu_n^2(f)\nonumber\\
&\hspace{-4cm}\leq& \hspace{-2cm}2\eta\left\|\widehat{\boldsymbol\beta}_{\boldsymbol\lambda,\widehat m}-\boldsymbol\beta^*\right\|_n^2+ 2\eta\left\|\widehat{\boldsymbol\beta}_{\boldsymbol\lambda,m}-\boldsymbol\beta^*\right\|_n^2+\frac{\kappa}2\log(n)\sigma^2\frac{\max\{m,\widehat m\}}n.\label{eq:oracle3}
\end{eqnarray}

Gathering equations~\eqref{eq:oracle1} and~\eqref{eq:oracle3}, we get, on the set $\mathcal A\cap\mathcal B_m$,
\begin{eqnarray*}
(1-2\eta)\left\|\widehat{\boldsymbol\beta}_{\boldsymbol\lambda,\widehat m}-\boldsymbol\beta^*\right\|_n^2&\leq&(1+ 2\eta)\left\|\widehat{\boldsymbol\beta}_{\boldsymbol\lambda,m}-\boldsymbol\beta^*\right\|_n^2\\
&&+\frac\kappa2\log(n)\sigma^2\frac{\max\{m,\widehat m\}}{n}+\kappa\log(n)\sigma^2\frac{m}{n}-\kappa\log(n)\sigma^2\frac{\widehat m}{n}.\\
&\leq &(1+ 2\eta)\left\|\widehat{\boldsymbol\beta}_{\boldsymbol\lambda,m}-\boldsymbol\beta^*\right\|_n^2+2\kappa\log(n)\sigma^2\frac{m}{n}.
\end{eqnarray*}

and the quantity $\left\|\widehat{\boldsymbol\beta}_{\boldsymbol\lambda,m}-\boldsymbol\beta^*\right\|_n^2$ is upper-bounded in Proposition~\ref{prop:oracle}. We obtain the expected result with $\tilde\eta=(1+2\eta)/(1-2\eta)-1>0$. 

To conclude, since it has already been proven in Proposition~\ref{prop:oracle} that $\mathbb P(\mathcal A^c)\leq p^{1-q}$, it remains to prove that there exists a constant $C_{MS}>0$ such that
\[
\mathbb P\left(\cup_{m=1}^m\mathcal B_m^c\right)\leq \frac{C_{MS}}n. 
\]
We have
\[
\mathbb P\left(\cup_{m=1}^m\mathcal B_m^c\right)\leq\sum_{m=1}^{N_n}\sum_{m'=1}^{N_n}\mathbb P\left(\sup_{f\in\mathbf H^{(\max\{m,m'\})},\|f\|_n=1}\nu_n^2(f)\geq\frac\kappa{2\eta^{-1}}\log(n)\sigma^2\frac{\max\{m,m'\}}{n}\right). 
\]
We apply~Lemma~\ref{lem:control_nun_empir} with $t=\left(\frac\kappa{2\eta^{-1}}\log(n)-1\right)\sigma^2\frac{\max\{m,m'\}}{n}\leq\frac\kappa6\log(n)\sigma^2\frac{\max\{m,m'\}}{n}$ and obtain
\begin{align*}
\mathbb P\left(\sup_{f\in\mathbf H^{(\max\{m,m'\})},\|f\|_n=1}\nu_n^2(f)\geq\frac\kappa6\log(n)\sigma^2\frac{\max\{m,m'\}}{n}\right)&\\
&\hspace{-8cm}\leq\exp\left(-2\kappa\log(n)\max\{m,m'\}\min\left\{\frac{\kappa\log(n)}{6912},\frac1{1536}\right\}\right). 
\end{align*}
Suppose that $\kappa\log(n)>6912/1536=9/2$ (the other case could be treated similarly), we have, since $1\leq m\leq N_n\leq n$, and by bounding the second sum by an integral
\[
\begin{split}
\sum_{m=1}^{N_n}\sum_{m'=1}^{N_n}\mathbb P\left(\sup_{f\in\mathbf H^{(\max\{m,m'\})},\|f\|_n=1}\nu_n^2(f)\geq\frac\kappa6\log(n)\sigma^2\frac{\max\{m,m'\}}{n}\right)&
\\&\hspace{-9cm}\leq \sum_{m=1}^{N_n}\left(\sum_{m'=1}^m\exp\left(-\frac{\kappa\log(n)m}{768}\right)+\sum_{m'=m+1}^{N_n}\exp\left(-\frac{\kappa\log(n)m'}{768}\right)\right)
\\&\hspace{-9cm}\leq N_n\ n\exp\left(-\frac{\kappa\log(n)}{768}\right)+\frac{768N_n}{\kappa\log(n)}\exp\left(-\frac{\kappa\log(n)}{768}\right). 
\end{split}
\]
Now choosing $\kappa>2304$ we know that there exists a universal constant $C_{MS}>0$ such that
\[
\mathbb P\left(\cup_{m=1}^{N_n}\mathcal B_m^c\right)\leq C_{MS}/n.
\]

Note that the minimal value 2304 for $\kappa$ is purely theoretical and does not correspond to a value of $\kappa$ which can reasonably be used in practice. 
 \end{proof}

\subsection{Proof of Theorem~\ref{thm:oracle_pred}}
\label{proof:thm:oracle_pred}
\begin{proof}
In the proof, the notation $C,C',C''>0$ denotes quantities which may vary from line to line but are always independent of $n$ or $m$. 

Let $\mathcal A$ the set defined in the statement of Proposition~\ref{prop:oracle} and $\mathcal B=\bigcap_{m=1}^{N_n}\mathcal B_m$ the set appearing in the proof of Theorem~\ref{thm:oracle} (see Equation~\eqref{eq:defBm} p.~\pageref{eq:defBm}). Following the proof of Theorem~\ref{thm:oracle}, we know that, on the set $\mathcal A\cap\mathcal B$, for all $m=1,\hdots,N_n,M_n$, for all $\boldsymbol\beta\in\mathbf H^{(m)}$ such that $J(\boldsymbol\beta)\leq s$, for all $\tilde\eta>0$,
\begin{equation}\label{eq:majo_normen}
\left\|\widehat{\boldsymbol\beta}_{\boldsymbol\lambda,\widehat m}-\boldsymbol\beta^*\right\|_n^2\leq (1+\tilde\eta)\|\boldsymbol\beta-\boldsymbol\beta^*\|_n^2+\frac{9}{4\left(\tilde\kappa_n^{(m)}(s)\right)^2}\sum_{j\in J(\boldsymbol\beta)}\lambda_j^2+2\kappa\log(n)\sigma^2\frac{m}n,
\end{equation}
since $C(\tilde\eta)\leq 2$.
We also now that 
\[
\mathbb P(\mathcal A^c)\leq p^{1-q}\text{ and }\mathbb P(\mathcal B^c)\leq \frac{C_{MS}}n.
\]

We define now the set $\mathcal C=\mathcal C_1\cap\mathcal C_2$ where 
\begin{eqnarray*}
\mathcal C_1&:=&\left\{\sup_{\boldsymbol\beta\in\mathbb H^{(N_n)}\backslash\{0\}}\left|\frac{\|\boldsymbol\beta\|_{n}^2}{\|\boldsymbol\beta\|_{\boldsymbol\Gamma}^2}-1\right|\leq\frac12\right\}\\
\mathcal C_2&:=&\left\{\sup_{\boldsymbol\beta\in\mathbb H^{(N_n)}\backslash\{0\}}\left|\frac{\|\boldsymbol\beta\|_{n}^2-\|\boldsymbol\beta\|_{\boldsymbol\Gamma}^2}{\|\boldsymbol\beta\|^2}\right|\leq\frac12\right\}.
\end{eqnarray*}
and prove that  \begin{equation}\label{eq:majo_Ccomp}
 \mathbb P(\mathcal C^c)\leq 2n^2\exp(-c_{\max}n), \quad c_{\max}=\max\left\{(4b {\rm tr}(\boldsymbol\Gamma)(4 {\rm tr}(\boldsymbol\Gamma)+1/2))^{-1};r_{\boldsymbol\Gamma}^2/16b\right\},
 \end{equation}
 where $r_{\boldsymbol\Gamma}>0$ depends only on $\boldsymbol\Gamma$ and is defined below.

We apply Proposition~\ref{prop:norm_equiv} to bound
\begin{eqnarray}
\mathbb P(\mathcal C_1^c)&=& \mathbb P\left(\sup_{\boldsymbol\beta\in\mathbb H^{(N_n)}\backslash\{0\}}\left|\frac{\|\boldsymbol\beta\|_{n}^2-\|\boldsymbol\beta\|_{\boldsymbol\Gamma}^2}{\|\boldsymbol\beta\|_{\boldsymbol\Gamma}^2}\right|>\frac12\right)\nonumber\\
&\leq& 2N_n^2\exp\left(-\frac{n\rho^2(\boldsymbol\Gamma_{|N_n})}{4b\sum_{j=1}^{N_n}\tilde v_j\left(4\sum_{j=1}^{N_n}\tilde v_j+\frac{\rho(\boldsymbol\Gamma_{|N_n})}2\right)}\right)\nonumber\\
&\leq& 2N_n^2\exp\left(-\frac{n\rho^2(\boldsymbol\Gamma_{|N_n})}{16b{\rm tr}^2(\boldsymbol\Gamma_{|N_n})}\right).\label{eq:majoCcomp2}
\end{eqnarray}
We remark that 
\[
\rho(\boldsymbol\Gamma_{|m})=\sup_{f\in\mathbb H^{(N_n)}\backslash\{0\}}\frac{\|\boldsymbol\Gamma_{|N_n}f\|}{\|f\|}\underset{m \to +\infty}{\longrightarrow}\rho(\boldsymbol\Gamma)  \text{ and }{\rm tr}(\boldsymbol\Gamma_{|m})=\sum_{k=1}^{N_n}\langle\boldsymbol\Gamma\boldsymbol\varphi^{(k)},\boldsymbol\varphi^{(k)}\rangle\underset{m \to +\infty}{\longrightarrow}{\rm tr}(\boldsymbol\Gamma),
\]
then 
\[
\frac{\rho(\boldsymbol\Gamma_{|m})}{{\rm tr}(\boldsymbol\Gamma_{|m})}\underset{m \to +\infty}{\longrightarrow}\frac{\rho(\boldsymbol\Gamma)}{{\rm tr}(\boldsymbol\Gamma)}>0,
\]
and there exists a constant $r_{\boldsymbol\Gamma}>0$ such that, for all $m$,
\[
\frac{\rho(\boldsymbol\Gamma_{|m})}{{\rm tr}(\boldsymbol\Gamma_{|m})}\geq r_{\boldsymbol\Gamma}.
\]
Then from Equation~\eqref{eq:majoCcomp2}, and the fact that $N_n\leq n$, we get that 
 \[
 \mathbb P(\mathcal C_1^c)\leq 4n^2\exp\left(-\frac{r_{\boldsymbol\Gamma}n}{16b}\right).
 \]

We turn now to the upper-bound on the probability of $\mathcal C_2^c$ and apply again Proposition~\ref{prop:norm_equiv} 
\begin{eqnarray*}
\mathbb P\left(\sup_{\boldsymbol\beta\in\mathbb H^{(N_n)}\backslash\{0\}}\left|\frac{\|\boldsymbol\beta\|_{n}^2-\|\boldsymbol\beta\|_{\boldsymbol\Gamma}^2}{\|\boldsymbol\beta\|^2}\right|>\frac12\right)&\leq &2N_n^2\exp\left(-\frac{n/4}{b\sum_{j=1}^{N_n}\tilde v_j\left(4\sum_{j=1}^{N_n}\tilde v_j+\frac12\right)}\right)\\
&\leq&2N_n^2\exp\left(-\frac{n}{4b{\rm tr}(\boldsymbol\Gamma)\left(4{\rm tr}(\boldsymbol\Gamma)+\frac12\right)}\right)
\end{eqnarray*}
and the upper-bound~\eqref{eq:majo_Ccomp} comes from the fact that
 \[
 \sum_{k=1}^{N_n}\tilde v_k=\sum_{k=1}^{N_n}\mathbb E[\langle\boldsymbol\varphi^{(k)},\mathbf X_1\rangle^2]= \sum_{k=1}^{N_n}\langle\boldsymbol\Gamma\boldsymbol\varphi^{(k)},\boldsymbol\varphi^{(k)}\rangle={\rm tr}(\boldsymbol\Gamma_{|N_n})\leq {\rm tr}(\boldsymbol\Gamma). 
 \]

 On the set $\mathcal A\cap\mathcal B\cap\mathcal C$, we have then, for all $m=1,\hdots,N_n$, 
  \begin{eqnarray*}
  \left\|\widehat{\boldsymbol\beta}_{\boldsymbol\lambda,\widehat m}-\boldsymbol\beta^*\right\|_{\boldsymbol\Gamma}^2&\leq& 2\left\|\widehat{\boldsymbol\beta}_{\boldsymbol\lambda,\widehat m}-\boldsymbol\beta^{(*,m)}\right\|_{\boldsymbol\Gamma}^2+ 2\left\|\boldsymbol\beta^{(*,m)}-\boldsymbol\beta^{*}\right\|_{\boldsymbol\Gamma}^2\\
&\leq& 4\left\|\widehat{\boldsymbol\beta}_{\boldsymbol\lambda,\widehat m}-\boldsymbol\beta^{(*,m)}\right\|_n^2+ 2\left\|\boldsymbol\beta^{(*,m)}-\boldsymbol\beta^{*}\right\|_{\boldsymbol\Gamma}^2
  \end{eqnarray*}
  From~\eqref{eq:majo_normen}, we get 
  \begin{eqnarray*}
\left\|\widehat{\boldsymbol\beta}_{\boldsymbol\lambda,\widehat m}-\boldsymbol\beta^*\right\|_{\boldsymbol\Gamma}^2&\leq &C\left(\|\boldsymbol\beta-\boldsymbol\beta^*\|_{n}^2+\frac1{\left(\tilde\kappa_n^{(m)}(s)\right)^2}\sum_{j\in J(\boldsymbol\beta)}\lambda_j^2+\kappa\frac{\log n}n\sigma^2m\right.\\
&&\quad\left.+\left\|\boldsymbol\beta^*-\boldsymbol\beta^{(*,m)}\right\|_{\boldsymbol\Gamma}^2\right),
  \end{eqnarray*}
  for a constant $C>0$. 
Now remark that, on the set $\mathcal C_1$
\[
\n{\bbeta^{(*,m)}-\bbeta}_n^2\leq \frac32\n{\bbeta^{(*,m)}-\bbeta}_\cov^2
\]
and, for all $J\subset \{1,\hdots,p\}$, for all $\boldsymbol\delta\in\mathbf H^{(m)}$,  such that $\sum_{j\in J}\|\delta_j\|_j^2\neq 0$,
 \[
 \frac12\frac{\|\boldsymbol\delta\|_{\boldsymbol\Gamma}^2}{\sqrt{\sum_{j\in J}\|\delta_j\|_j^2}}\leq\frac{\|\boldsymbol\delta\|_n^2}{\sqrt{\sum_{j\in J}\|\delta_j\|_j^2}}\leq  \frac32\frac{\|\boldsymbol\delta\|_{\boldsymbol\Gamma}^2}{\sqrt{\sum_{j\in J}\|\delta_j\|_j^2}},
 \]
 which implies 
  \[
 \frac12\kappa_n^{(m)}(s)\leq\tilde\kappa_n^{(m)}(s)\leq  \frac32\kappa_n^{(m)}(s). 
 \]
 We obtain 
   \begin{eqnarray}
\left\|\widehat{\boldsymbol\beta}_{\boldsymbol\lambda,\widehat m}-\boldsymbol\beta^*\right\|_{\Gamma}^2&\leq &C'\min_{m=1,\hdots,N_n}\min_{\boldsymbol\beta\in\mathbf H^{(m)},|J(\boldsymbol\beta)|\leq s}\left\{\|\boldsymbol\beta-\boldsymbol\beta^*\|_{\boldsymbol\Gamma}^2+\frac1{\left(\kappa_n^{(m)}(s)\right)^2}\sum_{j\in J(\boldsymbol\beta)}\lambda_j^2\right.\nonumber\\
&&\left.+\kappa\frac{\log n}n\sigma^2m+\left\|\boldsymbol\beta^*-\boldsymbol\beta^{(*,m)}\right\|_{\boldsymbol\Gamma}^2+\left\|\boldsymbol\beta^*-\boldsymbol\beta^{(*,m)}\right\|_n^2\right\}.\label{eq:oracle_pred1}
  \end{eqnarray}
  Now, let $\zeta_{n,m}=\frac{\log(n)}{\sqrt{n}}\|\boldsymbol\beta^{(*,\perp m)}\|$ and
  \[
  \mathcal D:=\bigcap_{m=1}^{N_n}\left\{\|\boldsymbol\beta^{(*,\perp m)}\|_n^2\leq \|\boldsymbol\beta^{(*,\perp m)}\|_{\boldsymbol\Gamma}^2+\zeta_{n,m}\right\},
  \]
  where we recall the notation $\boldsymbol\beta^{(*,\perp m)}=\boldsymbol\beta^*-\boldsymbol\beta^{(*,m)}$.   
  We give now an upper-bound on $\mathbb P(\mathcal D^c)$ which completes the proof. Remark that 
  \[
  \|\boldsymbol\beta^{(*,\perp m)}\|_n^2=\frac1n\sum_{i=1}^n\langle \boldsymbol\beta^{(*,\perp m)}, \mathbf X_i\rangle^{2},
  \]
  and that, for all $i=1,\hdots,n$, 
  \[
  \mathbb E\left[ \langle \beta^{(*,\perp m)}, \mathbf X_i\rangle^{2}\right]=\|\boldsymbol\beta^{(*,\perp m)}\|_{\boldsymbol\Gamma}^2,
  \]
we can rewrite 
 \[
  \mathcal D:=\bigcap_{m=1}^{N_n}\left\{\frac1n\sum_{i=1}^n\left(\langle \boldsymbol\beta^{(*,\perp m)}, \mathbf X_i\rangle^{2}- \mathbb E\left[ \langle \boldsymbol\beta^{(*,\perp m)}, \mathbf X_1\rangle^{2}\right]\right)\leq \zeta_{n,m}\right\}. 
  \]
  Hence 
  \[
  \mathbb P(\mathcal D^c)\leq\sum_{m=1}^{N_n}\mathbb P\left(\frac1n\sum_{i=1}^n\left(\langle\boldsymbol \beta^{(*,\perp m)}, \mathbf X_i\rangle^{2}- \mathbb E\left[ \langle\boldsymbol \beta^{(*,\perp m)}, \mathbf X_1\rangle^{2}\right]\right)> \zeta_{n,m}\right). 
  \]
  We upper-bound the quantities above using Bernstein's inequality (Proposition~\ref{lem:Bernstein}, p.~\pageref{lem:Bernstein}). 
  
  We have, for $\ell\geq 2$,
  \[
  \mathbb E\left[\langle\boldsymbol \beta^{(*,\perp m)}, \mathbf X_i\rangle^{2\ell}\rangle\right]\leq \|\boldsymbol \beta^{(*,\perp m)}\|^{2\ell}\mathbb E\left[\|\mathbf X_i\|^{2\ell}\right]\leq \frac{\ell !}2\|\boldsymbol \beta^{(*,\perp m)}\|^{2}v_{Mom}^2\left(\|\boldsymbol \beta^{(*,\perp m)}\|c_{Mom}\right)^{\ell-2},
  \]
  applying Bernstein inequality, we get 
  \[
  \mathbb P(\mathcal D_n^c)\leq \sum_{m=1}^{N_n}\exp\left(-\frac{n\zeta_{n,m}^2/2}{\|\boldsymbol\beta^{(*,\perp m)}\|^2v_{Mom}^2+\zeta_{n,m}\|\boldsymbol\beta^{(*,\perp m)}\|c_{Mom}}\right). 
  \]
  
  We get, since $N_n\leq n$,
  \[
  \mathbb P(\mathcal D_n^c)\leq N_n\exp\left(-\frac{\log^2(n)}{2v_{Mom}^2+\frac{\log(n)}{\sqrt{n}}c_{Mom}}\right)\leq \frac {C_{Mom}}n,
  \]
  with $C_{Mom}>0$ depending only on $v_{Mom}$ and $c_{Mom}$.
  
  Then, on $\mathcal A\cap\mathcal B\cap\mathcal C\cap\mathcal D$, \eqref{eq:oracle_pred1} becomes, for all $m=1,\hdots,N_n$,
   \begin{eqnarray*}
\left\|\widehat{\boldsymbol\beta}_{\boldsymbol\lambda,\widehat m}-\boldsymbol\beta^*\right\|_{\Gamma}^2&\leq &C\left(\|\boldsymbol\beta-\boldsymbol\beta^*\|_{\boldsymbol\Gamma}^2+\frac1{\left(\kappa_n^{(m)}(s)\right)^2}\sum_{j\in J(\boldsymbol\beta)}\lambda_j^2+\kappa\frac{\log n}n\sigma^2m\right.\\
&&\quad\left.+\left\|\boldsymbol\beta^{(*,\perp m)}\right\|_{\boldsymbol\Gamma}^2+\zeta_{n,m}\right). 
  \end{eqnarray*}
  We then upper-bound $\zeta_{n,m}$ as follows
  \[
  \zeta_{n,m}\leq \left(\kappa_n^{(m)}(s)\right)^2\left\|\boldsymbol\beta^{(*,\perp m)}\right\|^2+\frac{\log^2(n)}{n\left(\kappa_n^{(m)}(s)\right)^2}.
  \]
\end{proof}

\section{Control of empirical processes}
\begin{lemma}\label{lem:control_nun_empir}
For all $t>0$, for all $m$,
\begin{align*}
\mathbb P_{\mathbf X}\left(\sup_{\mathbf f\in\mathbf H^{(m)}, \|f\|_n=1}\left(\frac1n\sum_{i=1}^n\varepsilon_i\langle\mathbf f,\mathbf X_i\rangle\right)^2\geq \sigma^2\frac{m}{n}+t\right)\\
&\hspace{-6cm}\leq \exp\left(-\min\left\{\frac{n^2t^2}{1536\sigma^4m};\frac{nt}{512\sigma^2}\right\}\right),
\end{align*}
where $\mathbb P_{\mathbf X}(\cdot):=\mathbb P(\cdot|\mathbf X_1,\hdots,\mathbf X_n)$ is the conditional probability given $\mathbf X_1,\hdots,\mathbf X_n$.
\end{lemma}
\begin{proof}[Proof of Lemma~\ref{lem:control_nun_empir}]
We follow the ideas of \citet{B00}. Let $m$ be fixed, and 
\[
S_m:=\left\{x=(x_1,\hdots,x_n)^t\in\mathbb R^n, \exists \mathbf f\in\mathbb H^{(m)}, \forall i, x_i=\langle\mathbf f,\mathbf X_i\rangle\right\}.
\]
We known that $S_m$ is a linear subspace of $\mathbb R^n$ and that
\[
\sup_{\mathbf f\in\mathbf H^{(m)}, \|f\|_n=1}\frac1n\sum_{i=1}^n\varepsilon_i\langle\mathbf f,\mathbf X_i\rangle=\frac1n\sup_{x\in S_m, x^tx=n}\varepsilon^t x=\frac1{\sqrt{n}}\sup_{x\in S_m, x^tx=1}\varepsilon^t x=\frac1{\sqrt{n}}\sqrt{\varepsilon^tP_m\varepsilon},
\]
where $\varepsilon=(\varepsilon_1,\hdots,\varepsilon_n)^t$ and $P_m$ is the matrix of the orthogonal projection onto $S_m$. 
This gives us 
\[
\mathbb P_{\mathbf X}\left(\sup_{\mathbf f\in\mathbf H^{(m)}, \|f\|_n=1}\left(\frac1n\sum_{i=1}^n\varepsilon_i\langle\mathbf f,\mathbf X_i\rangle\right)^2\geq \sigma^2\frac{m}{n}+t\right)= \mathbb P_{\mathbf X}\left(\varepsilon^tP_m\varepsilon\geq \sigma^2m+nt\right).
\]
We apply now \citet[Theorem 3]{B19}, with $A=P_m$ and obtain the expected results, since 
\[
\mathbb E[\varepsilon^tP_m\varepsilon]=\sigma^2\text{tr}(P_m)=\sigma^2m, 
\]
and since the Frobenius norm $\|\cdot\|_F$ of $P_m$ is equal to $\|P_m\|_F=\sqrt{\text{tr}(\Pi_m^t\Pi_m)}=\sqrt{m}$ and its matrix norm $\|P_m\|_2=1$.
\end{proof}

\section{Tails inequalities}

\begin{proposition}{\bf Equivalence of tails of Banach-valued random variables  \citep[Equation~(3.5) p. 59]{LT91}.}
\label{prop:LT91}

\medskip 

Let $X$ be a Gaussian random variable in a Banach space $(B,\|\cdot\|)$. For every $t>0$, 
\[
\mathbb P\left(\|X\|>t\right)\leq 4\exp\left(-\frac{t^2}{8\mathbb E\left[\|X\|^2\right]}\right).
\] 
\end{proposition}

\bigskip

\begin{proposition}{\bf Bernstein inequality \citep[Lemma 8]{BM98}.}
\label{lem:Bernstein}

\medskip 

Let $Z_1,\hdots,Z_n$ be independent random variables satisfying the moments conditions
\[
\frac1n\sum^n_{i=1}\mathbb E\left[\left|Z_i\right|^\ell\right] \leq \frac{\ell!}2 v^2c^{\ell-2}, \text{ for all }\ell \geq 2,
\]
for some positive constants $v$ and $c$. Then, for any positive $\varepsilon$,
\[
\mathbb P\left(\left|\frac1n\sum^n_{i=1}Z_i - \mathbb E\left[Z_i\right]\right| \geq \varepsilon\right)\leq 2 \exp\left(-\frac{n\varepsilon^2/2}{v^2 + c\varepsilon}\right).
\]
\end{proposition}


%
%
%

\bigskip

\begin{proposition}{\bf Norm equivalence in finite subspaces.}
\label{prop:norm_equiv}

\medskip 

Let $\mathbf X_1,...,\mathbf X_n$ be i.i.d copies of a random variable $\mathbf X$ verifying Assumption $(H_{Mom}^{(1)})$.
Then, for all $t>0$, for all weights $\mathbf w=(w_1,...,w_m)\in]0,+\infty[^m$,
\begin{equation}\label{eq:equiv_norm_w}
\mathbb P\left(\sup_{\boldsymbol\beta\in\mathbf H^{(m)}\backslash\{0\}}\left|\frac{\|\boldsymbol\beta\|_n^2-\|\boldsymbol\beta\|_{\boldsymbol\Gamma}^2}{\|\boldsymbol\beta\|^2_{\mathbf w}}\right|>t\right)\leq 2m^2\exp\left(-\frac{nt^2}{b\sum_{j=1}^m\frac{\tilde v_j}{w_j}\left(4\sum_{j=1}^m\frac{\tilde v_j}{w_j}+ t\right)}\right),
\end{equation}
where $\|\boldsymbol\beta\|_n^2=\frac1n\sum_{i=1}^n\langle\boldsymbol\beta,\mathbf X_i\rangle^2$,  $\|\boldsymbol\beta\|_{\boldsymbol\Gamma}^2=\mathbb E\left[\|\boldsymbol\beta\|_n^2\right]$, and $\|\boldsymbol\beta\|_{\mathbf w}^2=\sum_{j=1}^mw_j\langle\boldsymbol\beta,\boldsymbol\varphi^{(j)}\rangle^2$ and
\begin{equation}\label{eq:equiv_norm_Gamma}
\mathbb P\left(\sup_{\boldsymbol\beta\in\mathbf H^{(m)}\backslash\{0\}}\left|\frac{\|\boldsymbol\beta\|_n^2-\|\boldsymbol\beta\|_{\boldsymbol\Gamma}^2}{\|\boldsymbol\beta\|^2_{\boldsymbol\Gamma}}\right|>t\right)\leq 2m^2\exp\exp\left(-\frac{n\rho^2(\boldsymbol\Gamma_{|m})t^2}{b\sum_{j=1}^m\tilde v_j\left(4\sum_{j=1}^m\tilde v_j+ t\rho(\boldsymbol\Gamma_{|m})\right)}\right).
\end{equation}
\end{proposition}

\begin{proof}[Proof of Proposition~\ref{prop:norm_equiv}]

We have, for all $\boldsymbol\beta\in\mathbf H^{(m)}$, $\|\boldsymbol\beta\|_n^2=\langle\widehat{\boldsymbol\Gamma}\boldsymbol\beta,\boldsymbol\beta\rangle$. Hence, 
\[
\|\boldsymbol\beta\|_n^2-\|\boldsymbol\beta\|_{\boldsymbol\Gamma}^2=\langle(\widehat{\boldsymbol\Gamma}-\boldsymbol\Gamma)\boldsymbol\beta,\boldsymbol\beta\rangle=\sum_{j,k=1}^m\langle\boldsymbol\beta,\boldsymbol\varphi^{(j)}\rangle\langle\boldsymbol\beta,\boldsymbol\varphi^{(k)}\rangle \langle(\widehat{\boldsymbol\Gamma}-\boldsymbol\Gamma)\boldsymbol\varphi^{(j)},\boldsymbol\varphi^{(k)}\rangle=b^t\Phi_m b,
\]
with $b:=\left(\langle\boldsymbol\beta,\boldsymbol\varphi^{(1)}\rangle,...,\langle\boldsymbol\beta,\boldsymbol\varphi^{(m)}\rangle\right)^t$ and $\Phi_m=\left(\langle(\widehat{\boldsymbol\Gamma}-\boldsymbol\Gamma)\boldsymbol\varphi^{(j)},\boldsymbol\varphi^{(k)}\rangle\right)_{1\leq j,k\leq m}$ which implies
\begin{eqnarray*}
\sup_{\boldsymbol\beta\in\mathbf H^{(m)}\backslash\{0\}}\left|\frac{\|\boldsymbol\beta\|_n^2-\|\boldsymbol\beta\|_{\boldsymbol\Gamma}^2}{\|\boldsymbol\beta\|_{\mathbf w}^2}\right|&=&\rho(W^{-1/2}\Phi_mW^{-1/2})\leq\sqrt{ \text{tr}(W^{-1}\Phi_m\Phi_m^tW^{-1})}\\
&=&\sqrt{\sum_{j,k=1}^m\frac{\langle(\widehat{\boldsymbol\Gamma}-\boldsymbol\Gamma)\boldsymbol\varphi^{(j)},\boldsymbol\varphi^{(k)}\rangle^2}{w_jw_k}},
\end{eqnarray*}
where $\rho$ denotes the spectral radius, and $W$ the diagonal matrix with diagonal entries $(w_1,\hdots,w_m)$.
We then have 
\begin{eqnarray*}
\mathbb P\left(\sup_{\boldsymbol\beta\in\mathbf H^{(m)}\backslash\{0\}}\left|\frac{\|\boldsymbol\beta\|_n^2-\|\boldsymbol\beta\|_{\boldsymbol\Gamma}^2}{\|\boldsymbol\beta\|^2_{\mathbf w}}\right|>t\right)&\leq&\mathbb P\left(\sum_{j,k=1}^m\frac{\langle(\widehat{\boldsymbol\Gamma}-\boldsymbol\Gamma)\boldsymbol\varphi_j,\boldsymbol\varphi_k\rangle^2}{w_jw_k}>t^2\right)\\
&\leq & \mathbb P\left(\bigcup_{j,k=1}^m\left\{\frac{\langle(\widehat{\boldsymbol\Gamma}-\boldsymbol\Gamma)\boldsymbol\varphi^{(j)},\boldsymbol\varphi^{(k)}\rangle^2}{w_jw_k}>p_{j,k}t^2\right\}\right),\\
&\leq & \sum_{j,k=1}^m\mathbb P\left(\frac{\left|\langle(\widehat{\boldsymbol\Gamma}-\boldsymbol\Gamma)\boldsymbol\varphi^{(j)},\boldsymbol\varphi^{(k)}\rangle\right|}{\sqrt{w_jw_k}}>\sqrt{p_{j,k}}t\right),
\end{eqnarray*}
where $p_{j,k}:=\frac{\tilde v_j\tilde v_k}{w_jw_k}\left(\sum_{\ell=1}^m\tilde v_\ell/w_\ell\right)^{-2}$ (remark that $\sum_{j,k=1}^mp_{j,k}=1$). Now, for all $j,k=1,...,m$,
\begin{eqnarray*}
\mathbb P\left(\frac{\left|\langle(\widehat{\boldsymbol\Gamma}-\boldsymbol\Gamma)\boldsymbol\varphi^{(j)},\boldsymbol\varphi^{(k)}\rangle\right|}{\sqrt{w_jw_k}}>\sqrt{p_{j,k}}t\right)&&\\
&&\hspace{-3cm}=\mathbb P\left(\left|\frac1n\sum_{i=1}^n\frac{\langle\boldsymbol\varphi^{(j)},\mathbf X_i\rangle\langle\boldsymbol\varphi^{(k)},\mathbf X_i\rangle}{\sqrt{w_jw_k}}-\mathbb E\left[\frac{\langle\boldsymbol\varphi^{(j)},\mathbf X_i\rangle\langle\boldsymbol\varphi^{(k)},\mathbf X_i\rangle}{\sqrt{w_jw_k}}\right]\right|>\sqrt{p_{j,k}}t\right).
\end{eqnarray*}
By Cauchy-Schwarz inequality, for all $\ell\geq 2$,
\begin{eqnarray*}
\mathbb E\left[\left|\frac{\langle\boldsymbol\varphi^{(j)},\mathbf X_i\rangle\langle\boldsymbol\varphi^{(k)},\mathbf X_i\rangle}{\sqrt{w_jw_k}}\right|^\ell\right]&\leq& \frac{\sqrt{\mathbb E\left[\langle\boldsymbol\varphi^{(j)},\mathbf X\rangle^{2\ell}\right]\mathbb E\left[\langle\boldsymbol\varphi^{(k)},\mathbf X\rangle^{2\ell}\right]}}{\sqrt{w_jw_k}^\ell}\\
&\leq &\ell!b^{\ell-1}\sqrt{\frac{\tilde v_j}{w_j}}^{\ell}\sqrt{\frac{\tilde v_k}{w_k}}^{\ell}=\frac{\ell!}2\; 2b\frac{\tilde v_j}{w_j}\frac{\tilde v_k}{w_k}\;\left(b\sqrt{\frac{\tilde v_j}{w_j}}\sqrt{\frac{\tilde v_k}{w_k}}\right)^{\ell-2}.
\end{eqnarray*}
Hence, Bernstein's inequality (Lemma~\ref{lem:Bernstein}) implies that
\begin{equation*}
\mathbb P\left(\frac{\left|\langle(\widehat{\boldsymbol\Gamma}-\boldsymbol\Gamma)\boldsymbol\varphi^{(j)},\boldsymbol\varphi^{(k)}\rangle\right|}{\sqrt{w_jw_k}}>\sqrt{p_{j,k}}t\right)\leq 2\exp\left(-\frac{np_{j,k}t^2/2}{2b\frac{\tilde v_j\tilde v_k}{w_jw_k}+b\sqrt{\frac{\tilde v_j}{w_j}}\sqrt{\frac{\tilde v_k}{w_k}}\sqrt{p_{j,k}}t}\right),
\end{equation*}
and the definition of $p_{j,k}$ implies Equation~\eqref{eq:equiv_norm_w}. 

We proceed similarly to prove Equation~\eqref{eq:equiv_norm_Gamma} from the upper-bound 
\begin{equation*}
\sup_{\boldsymbol\beta\in\mathbf H^{(m)}\backslash\{0\}}\left|\frac{\|\boldsymbol\beta\|_n^2-\|\boldsymbol\beta\|_{\boldsymbol\Gamma}^2}{\|\boldsymbol\beta\|_{\boldsymbol\Gamma}^2}\right|=\rho(\boldsymbol\Gamma_{|m}^{-1/2}\Phi_m\boldsymbol\Gamma_{|m}^{-1/2})\leq\rho(\Phi_m)\rho(\boldsymbol\Gamma_{|m}^{-1})=\rho(\Phi_m)\rho(\boldsymbol\Gamma_{|m})^{-1}
\end{equation*}

Following the same reasoning as above with $w_1=\hdots=w_m=1$, we get, for all $t>0$,
\[
\mathbb P\left(\rho(\Phi_m)>t\right)\leq 2m^2\exp\left(-\frac{nt^2}{b\sum_{j=1}^m\tilde v_j\left(4\sum_{j=1}^m\tilde v_j+t\right)}\right),
\]
which proves Equation~\eqref{eq:equiv_norm_Gamma}.
\end{proof}

\bibliographystyle{abbrvnat}
\bibliography{biblio_lasso} 

\end{document}